\documentclass[11pt]{amsart}
\usepackage[T1]{fontenc}
\usepackage[utf8]{inputenc}
\usepackage{lmodern}
\usepackage{mathrsfs}
\usepackage{amsfonts}
\usepackage{latexsym,amsmath,amssymb}

 \renewcommand{\epsilon}{\varepsilon}

\newtheorem{theorem}{Theorem}[section]

 \newtheorem{lemma}[theorem]{Lemma}
 
 \newtheorem{Corollary}[theorem]{Corollary}
 \newtheorem{proposition}[theorem]{proposition}
 
\newtheorem{deff}[theorem]{Definition}
 \newtheorem{rem}[theorem]{Remark}
 \newcommand{\bth}{\begin{theorem}}
 \newcommand{\ble}{\begin{lemma}}
 \newcommand{\bcor}{\begin{corr}}
 \newcommand{\bdeff}{\begin{deff}}
 \newcommand{\bprop}{\begin{proposition}}
 \newcommand{\ele}{\end{lemma}}
 \newcommand{\ecor}{\end{corr}}
 \newcommand{\edeff}{\end{deff}}
 
 \newcommand{\eprop}{\end{proposition}}

 \renewcommand{\Pi}{\varPi}

 \renewcommand{\epsilon}{\varepsilon}

\numberwithin{equation}{section}

\thanks{The authors were supported  by National Science Foundation of China(No.11771103) and  Guangxi Natural Science Foundation(No.2017GXNSFFA198017).}

\title
[Nonlinear second boundary conditions ]{On the Dirichlet problem for special Lagrangian curvature potential equation}
\author{Rongli Huang}
\address{School of Mathematics and Statistics, Guangxi Normal University,
Guilin, Guangxi 541004, People's Republic of China}
\email{ronglihuangmath@gxnu.edu.cn}
 \author{Yongmei Liang}
\address{School of Mathematics and Statistics, Guangxi Normal University,
Guilin, Guangxi 541004, People's Republic of China}
\email{}
\date{}

\begin{document}
\maketitle
\begin{abstract}
In this paper, we study a class of special Lagrangian curvature potential equations and obtain  the existence of smooth solutions for Dirichlet
problem.  The existence result is based on a priori estimates of global $C^{0}$, $C^{1}$ and $C^{2}$ norms of solutions  under the assumption of existence of a subsolution.
\end{abstract}

\let\thefootnote\relax

\section{Introduction}
It is well known that the Dirichlet problem is well posed for elliptic equations. The situation is complicated for  fully nonlinear elliptic equations; see for instance the seminal papers of  L. Caffarelli. L. Nirenberg and J. Spruck \cite{CNS1}-\cite{CNS4}. We will restrict our attention from now on to special Lagrangian equation and some related problems. R. Harvey and H.B. Lawson established the theory of calibrated geometry and introduced the special Lagrangian  equation in \cite{HL} back in 1982. They obtained that its solutions $u$  have the property that the graph $p=\nabla u$ in $\mathbb{R}^{n}\times \mathbb{R}^{n} = \mathbb{C}^{n}$ being a Lagrangian submanifold which is absolutely volume-minimizing and the linearization at any solution being elliptic.  Let $x=(x_{1},x_{2},\cdots,x_{n})$, $u=u(x)$ and $\lambda(D^2 u)=(\lambda_1,\cdots, \lambda_n)$ are the eigenvalues of Hessian matrix $D^2 u$ according to $x$. Yuan \cite{Yuan}  proved that every
classical convex solution of  special Lagrangian equation
\begin{equation*}
\sum_{i=1}^n\arctan\lambda_{i}=c,\quad
 \mathrm{in}\quad \mathbb{R}^{n},
\end{equation*}
must be a quadratic polynomial. The second boundary value problem for special Lagrangian equation was considered by Brendle and Warren \cite{SM}. They obtained the existence and uniqueness of the solution by the elliptic method. Later Brendle-Warren' theorem was generalized by several authors (see, e.g., \cite{HO}-\cite{HR}).

Assume that

(A): $\Omega$ is $C^{4}$ bounded domain in $\mathbb{R}^{n}$, $\varphi: \partial\Omega\longrightarrow\mathbb{ R}$ be in $C^{4}(\partial\Omega)$ and $h: \bar{\Omega}\longrightarrow [(n-2)\frac{\pi}{2}+\delta, n\frac{\pi}{2})$ be in $C^{2}(\bar{\Omega})$ where $\delta>0$ small enough.

In another case, Collins,  Sebastien and  Xuan \cite{CPW}  were concerned with
Lagrangian phase equation
\begin{equation}\label{e1.1}
\sum_{i=1}^n\arctan\lambda_{i}=h(x),\quad
 \mathrm{in}\quad \Omega,
\end{equation}
associated with  Dirichlet problem
\begin{equation}\label{e1.2}
u=\varphi,\quad
 \mathrm{on}\quad \partial\Omega.
\end{equation}
Let $\alpha$ be any constant satisfying $0<\alpha<1$. They proved that
\begin{theorem}\label{t1.1}
Suppose that $\Omega$, $\varphi$ and $h$ satisfy (A). If there exists a $C^{4}(\bar{\Omega})$ subsolution of (\ref{e1.1}) and (\ref{e1.2}), i.e. there exists a function $\underline{u}\in C^{4}(\bar{\Omega})$
such that
\begin{equation*}
\left\{ \begin{aligned}
\sum_{i=1}^n\arctan\underline{\lambda}_{i}&\geq h(x), &&\quad
 \mathrm{in}\quad \Omega, \\
\underline{u}&=\varphi,             &&\quad
 \mathrm{on}\quad \partial\Omega,
\end{aligned} \right.
\end{equation*}
where $\underline{\lambda}_{i}$ are the eigenvalues of $D^{2}\underline{u}$. Then  the Dirichlet problem (\ref{e1.1}) (\ref{e1.2})
admits a unique $C^{3,\alpha}(\bar{\Omega})$ solution. In addition, if   $\Omega,h,\varphi$ and $\underline{u}$ are smooth, then the solution $u$ is smooth.
\end{theorem}
Here we noticed that they used the idea of B.Guan \cite{B} that the curvature assumption on $\partial\Omega$ is replaced by the assumption that there exists an admissible  subsolution of (\ref{e1.1}) and (\ref{e1.2}) to obtain the global $C^{2}$ estimates for the admissible solutions.

The aim of this paper is to study special Lagrangian curvature potential equation
\begin{equation}\label{e1.3}
\sum_{i=1}^n\arctan\kappa_{i}=h(x),\quad
 \mathrm{in}\quad \Omega,
\end{equation}
in conjunction with Dirichlet condition
\begin{equation}\label{e1.4}
u=\varphi,\quad
 \mathrm{on}\quad \partial\Omega.
\end{equation}
where $(\kappa_1,\cdots, \kappa_n)$ are the principal curvatures of the graph $\Gamma=\{(x,u(x))\mid x\in \Omega\}$.
 It shows that a very nice geometric interpretation of special Lagrangian curvature operator
$F(\kappa_{1},\cdots, \kappa_{n})\triangleq\sum_{i=1}^n\arctan\kappa_{i}$ in \cite{G}.
The background of the so-call special Lagrangian curvature potential  equation (\ref{e1.3}) can be seen in the literature \cite{G} \cite{HL1}.  In addition,  the first author and Sitong Li  considered (\ref{e1.3})  with the second boundary condition
\begin{equation}\label{e1.4a}
Du(\Omega)=\tilde{\Omega},
\end{equation}
then they obtained the following result \cite{HL2} for $h(x)$ being some constant.
\begin{theorem}\label{t1.1aa}
Suppose that  $\Omega$, $\tilde{\Omega}$  are uniformly convex bounded
domains with smooth boundary in $\mathbb{R}^{n}$. Then there exist a uniformly convex solution $u\in C^{\infty}(\bar{\Omega})$  and a unique constant $c$ with $h(x)=c$ solving (\ref{e1.3}) and (\ref{e1.4a}), and $u$ is unique up to a constant.
\end{theorem}

It's interested in considering Dirichlet problem (\ref{e1.3}) and (\ref{e1.4}) with the same geometric conditions on $\Omega$ as the above theorem.
The central result of this paper is to find solution under the assumptions of the existence of the subsolution, which originate from
a series of Bo Guan's paper (cf. \cite{B} \cite{B1} \cite{B2}).

Now we state our main result in the following which gives a generalization of  Theorem 1.1.
\begin{theorem}\label{t1.3}
Suppose that $\Omega$, $\varphi$ and $h$ satisfy (A). If there exists a $C^{4}(\bar{\Omega})$ subsolution of (\ref{e1.3}) and (\ref{e1.4}), i.e. there exists a function $\underline{u}\in C^{4}(\bar{\Omega})$
such that
\begin{equation}\label{e2.6}
\left\{ \begin{aligned}
\sum_{i=1}^n\arctan\underline{\kappa}_{i}&\geq h(x),\ \ &&\quad
 \mathrm{in}\quad \Omega, \\
\underline{u}&=\varphi,             &&\quad
 \mathrm{on}\quad \partial\Omega,
\end{aligned} \right.
\end{equation}
where $\underline{\kappa}_{i}$ are the principal curvatures of the graph $\underline{\Gamma}=\{(x,\underline{u}(x))\mid x\in \Omega\}$. Then   Dirichlet problem (\ref{e1.3}) (\ref{e1.4})
admits a unique $C^{3,\alpha}(\bar{\Omega})$ solution. In addition, if $\Omega,h,\varphi$ and $\underline{u}$ are smooth, then the solution $u$ is smooth.
\end{theorem}\par

This article is dived into several sections. In section two, we  state some well known algebraic and inequalities according to the structure conditions for the operator $F$. In section three, we can derive  the  zeroth and first order estimates for the admissible solutions. In section four, we  exhibit  how to estimate the second order derivatives of $u$ by the assumption of the  boundedness  for them on $\partial\Omega$.  In section five, we  establish  the desired bound for the second order derivatives of $u$ on $\partial\Omega$ which is a key and difficult part in the paper. Finally we give the proof of  Theorem  1.3 by the continuity method.

\section{Preliminaries}
In this section we collect various preliminary knowledge and auxiliary lemmas that  used in the proof of Theorem 1.3.\par
At the beginning,  we recall some necessary geometric quantities associated with the graph of a function $u\in C^{2}(\Omega)$.
We use the Einstein summation convention,
if the indices are different from 1 and $n$. $u_{i}=D_{i}u, u_{ij}=D_{ij}u,u_{ijk}=D_{ijk}u, \cdots$ Denote the  all derivatives of $u$ according to $x_{i},x_{j},x_{k},\cdots$.
In the coordinate system induced by the embedding $x\longmapsto(x,u(x))$ the metric of graph $u$ is given by
\begin{equation*}
g_{ij}=\delta_{ij}+u_{i}u_{j}
\end{equation*}
which is called the first fundamental form and its inverse is
\begin{equation}\label{e2.1}
g^{ij}=\delta_{ij}-\frac{u_{i}u_{j}}{1+|Du|^{2}}.
\end{equation}
The second fundamental form is given by
\begin{equation}\label{e2.2}
h_{ij}=\frac{u_{ij}}{\sqrt{1+|Du|^{2}}}.
\end{equation}
Using (\ref{e2.1}) and (\ref{e2.2}), we obtain
\begin{equation*}
h_{ik}g^{kj}=\frac{u_{ik}}{\sqrt{1+|Du|^{2}}}\Big(\delta_{kj}-\frac{u_{k}u_{j}}{1+|Du|^{2}}\Big).
\end{equation*}

From \cite{CNS5} and \cite{J}, the principal curvatures of the graph of $u$ are the eigenvalues of $[h_{ij}]$ relative to $[g_{ij}]$, thus they are the eigenvalues of the generally nonsymmetric matrix $[h_{ik}g^{kj}]$. Equivalently they are the eigenvalues of the symmetric matrix
\begin{equation*}
a_{ij}=b^{ik}h_{kl}b^{lj}=\frac{1}{w}b^{ik}u_{kl} b^{lj},
\end{equation*}
where $w=\sqrt{1+|Du|^{2}}$ and $[b^{ij}]$ is the positive square root of $[g^{ij}]$,  i.e. $b^{ik}b^{kj}=g^{ij}$, given by
\begin{equation*}
b^{ij}=\delta_{ij}-\frac{u_{i}u_{j}}{w(1+w)}.
\end{equation*}
Therefore
\begin{equation}\label{e2.3}
u_{ij}=wb_{ik}a_{kl}b_{lj},
\end{equation}
where $[b_{ij}]$ is the inverse of $[b^{ij}]$, given by
\begin{equation*}
b_{ij}=\delta_{ij}+\frac{u_{i}u_{j}}{1+w}.
\end{equation*}
Explicitly we see
\begin{equation}\label{e2.4}
a_{ij}=\frac{1}{w}\Big\{u_{ij}-\frac{u_{i}u_{l}u_{jl}}{w(1+w)}-\frac{u_{j}u_{l}u _{il}}{w(1+w)}+\frac{u_{i}u_{j}u_{k}u_{l}u_{kl}}{w^{2}(1+w)^{2}}\Big\}.
\end{equation}

Next we exhibit  some properties of the function $F$, which can be found in \cite{YY} and \cite{J}. As mentioned in the introduction, the function $F$ was defined by
\begin{equation}\label{e2.5}
F(\mathcal{A})=\sum_{i=1}^{n}\arctan \kappa_{i},
\end{equation}
where $(\kappa_1,\cdots, \kappa_n)$ are the eigenvalues of $\mathcal{A}=[a_{ij}]$.
The function $F$ satisfies
\begin{equation*}
F^{ij}(\mathcal{A})\eta_{i}\eta_{j}>0,\quad \eta\in\mathbb{R}^{n}-\{0\},
\end{equation*}
where
\begin{equation*}
F^{ij}(\mathcal{A})=\frac{\partial F(\mathcal{A})}{\partial a_{ij}}.
\end{equation*}
Note that $[F^{ij}]$ is symmetric because $\mathcal{A}$ is symmetric. Furthermore, $[F^{ij}]$ is diagonal if $\mathcal{A}$ is diagonal. \par

As a matter of convenience, we use the notation
\begin{equation*}
\underline{a}_{ij}=\frac{1}{\underline{w}}\Big\{\underline{u}_{ij}-\frac{\underline{u}_{i}\underline{u}_{l}\underline{u}_{jl}}{\underline{w}(1+\underline{w})}-\frac{\underline{u}_{j}\underline{u}_{l}\underline{u} _{il}}{\underline{w}(1+\underline{w})}+\frac{\underline{u}_{i}\underline{u}_{j}\underline{u} _{k}\underline{u}_{l}\underline{u}_{kl}}{\underline{w}^{2}(1+\underline{w})^{2}}\Big\}.
\end{equation*}
where $\underline{w}=\sqrt{1+|D\underline{u}|^{2}}$.

Since the technique in \cite{CPW} that we will develop to prove Theorem \ref{t1.3}  also works for our setting,
then  we will use certain properties of $F$ at various of points in the proof. These are summarized in the following lemmas, which was proved in \cite{CPW}.

\begin{lemma}\label{t2.3}
Let $0<\delta<\frac{\pi}{2}$. Assume that $\kappa_{1}\geq\kappa_{2}\geq\cdots\geq\kappa_{n}$  satisfy
 \begin{equation*}
 \sum_{i=1}^{n}\arctan \kappa_{i}\geq(n-2)\frac{\pi}{2}+\delta.
 \end{equation*}
 Then there hold \\
(i) $\kappa_{1}\geq\kappa_{2}\geq\cdots\geq\kappa_{n-1}>0$, $|\kappa_{n}|\leq\kappa_{n-1}$;\\
(ii) $\sum_{i=1}^{n}\kappa_{i}\geq0$;\\
\underline{}(iii) $\kappa_{n}\geq-\frac{1}{\tan(\delta)}$;\\
(iv) if $\kappa_{n}<0$, then $\sum_{i=1}^{n}\frac{1}{\kappa_{i}}\leq-\tan(\delta)$;\\
(v) For any $\sigma\in((n-2)\frac{\pi}{2},n\frac{\pi}{2})$,
\begin{equation*}
\Gamma^{\sigma}=\{\kappa\in\mathbb{R}^{n}:\sum_{i=1}^{n}\arctan \kappa_{i}>\sigma\}.
\end{equation*}
 is a convex set, and $\partial\Gamma^{\sigma}$ is a smooth convex hypersurface.
\end{lemma}

\begin{lemma}\label{t2.4}
 Let $\Omega$, $\varphi$ and $h$ satisfy (A), $u\in C^{3,\alpha}(\bar{\Omega})$ satisfy the Dirichlet problem (\ref{e1.3}) and (\ref{e1.4}). Suppose that there exists a function $\underline{u}$ such that for each point $x\in\Omega$ and each index $i$,
there holds
\begin{equation}\label{e2.8}
\lim_{t\rightarrow\infty}F(\underline{\kappa}+te_{i})> h(x),
\end{equation}
where $\underline{\kappa}_{i}$ are the principal curvatures of the graph $\underline{\Gamma}=\{(x,\underline{u}(x))|x\in\Omega\}$, $\underline{\kappa}=(\underline{\kappa}_{1},\underline{\kappa}_{2},\cdots,\underline{\kappa}_{n})$, $e_{i}$ is the $i$th standard basis vector and $F(\kappa)=\sum_{i=1}^{n}\arctan \kappa_{i}$.

Then there are constants $R_{0}$, $C_{1}>0$ to arrive at the following inequalities. If $|\kappa|\geq R_{0}$,  we either have
\begin{equation}\label{e2.9}
\sum_{i,j=1}^{n}F^{ij}(\mathcal{A})\big[\underline{a}_{ij}-a_{ij}\big]>C_{1}\sum_{p=1}^{n}F^{pp}(\mathcal{A}),
\end{equation}
or
 \begin{equation}\label{e2.10}
F^{ii}(\mathcal{A})>C_{1}\sum_{p=1}^{n}F^{pp}(\mathcal{A})
\end{equation}
 for each $i$.
\end{lemma}
\begin{rem}
Without loss of generality, in the following we set $C_{1}$, $C_{2}$, $\ldots$ to be  positive constants depending only on the known data.
\end{rem}
Introducing Corollary 3.2 in \cite{CPW}, we infer that
\begin{lemma}\label{t2.5}
 Assume that $\Omega$, $\varphi$ and $h$ satisfy $(A)$, $\underline{u}$ is a subsolution satisfying (\ref{e2.6}), $u\in C^{3,\alpha}(\bar{\Omega})$ is an admissible solution to (\ref{e1.3}) and (\ref{e1.4}). As before, let $\kappa=(\kappa_1,\cdots, \kappa_n)$ are the principal curvatures of the graph $\Gamma=\{(x,u(x))\mid x\in \Omega\}$. Then there exsits $R_{0}$ depending only on $\underline{u}$ and $\delta$, such that for any $|\kappa|\geq R_{0}$, we have
\begin{equation}\label{e2.11}
F^{ij}(\mathcal{A})[\underline{a}_{ij}-a_{ij}]\geq\tau>0,
\end{equation}
where $\tau$ is a constant depending on only $\underline{u}$ and $\delta$.
\end{lemma}

\section{ Zeroth and first order estimates}
In this section, we will prove the zeroth and first order estimates for Dirichlet problem (\ref{e1.3}) and (\ref{e1.4}).    The two-sided bound for $u$ can be directly inferred from the existence of the sub-solution and super-solution.\par
To see this, it is enough to consider  $\overline{u}$ satisfying
  \begin{equation}\label{e3.1}
\left\{ \begin{aligned}
\triangle  \overline{u}&= 0,\ \ &&\quad
 \mathrm{in}\quad \Omega, \\
 \overline{u}&=\varphi,  &&\quad \mathrm{on}\quad  \partial\Omega.
\end{aligned} \right.
\end{equation}
Review  the identity: $u_{ij}=wb_{ik}a_{kl}b_{lj}$. Let $D^{2}u=[u_{ij}]$, $\mathcal{B}=[b_{ij}]$, $\mathcal{A}=[a_{ij}]$, where $\mathcal{B}^{T}=\mathcal{B}$.
 By Lemma 2.3 we see that $\mathrm{tr} \mathcal{A}=\sum_{i=1}^{n}\kappa_{i}\geq 0$.
Since $D^{2}u=w\mathcal{B}^{T}\mathcal{A}\mathcal{B}$ where  $w=\sqrt{1+|Du|^{2}}$,  then we arrive at  $$\mathrm{tr} D^{2}u=\mathrm{tr} (w\mathcal{B}^{T}\mathcal{A}\mathcal{B})\geq 0.$$
Next we have the following result.
\begin{lemma}\label{t3.1}
Suppose that $\Omega$, $\varphi$ and $h$ satisfy (A), $\underline{u}$ is a  admissible subsolution  defined  by (\ref{e2.6}),
  $u$ satisfy (\ref{e1.3}) and (\ref{e1.4}). Let $\overline{u}:\Omega\rightarrow\mathbb{R}$ be the function satisfying (\ref{e3.1}).
   Then we have
   \begin{equation*}
\underline{u}\leq u\leq \overline{u}, \quad
 \mathrm{in}\quad \overline{\Omega}.
\end{equation*}
\end{lemma}
\begin{proof} Note that  $\Delta u=\mathrm{tr} D^{2}u\geq 0=\triangle  \overline{u}$ and  $F(\mathcal{A})\leq F(\mathcal{\underline{A}})$ where
 $\underline{\mathcal{A}}=[\underline{a}_{ij}]$.
 Therefore, the conclusion is derived from the maximum principle.
\end{proof}
By Lemma \ref{t3.1}, we have the following Hopf Lemma, due to Li-Wang\cite{LW}.
\begin{lemma}\label{t3.1a}
Suppose that $\Omega$, $\varphi$ and $h$ satisfy (A), $\underline{u}$ is a  admissible subsolution  defined  by (\ref{e2.6}),
  $u$ satisfy (\ref{e1.3}) and (\ref{e1.4}). Suppose that $\hat{x}\in \partial\Omega$, $\hat{s}\in\mathbb{R}$, then we have
  $$\liminf _{\hat{s}\rightarrow0^{+}}\frac{(u-\underline{u})(\hat{x}+\hat{s} \nu(\hat{x}))}{\hat{s}}>0,$$
where $\nu$ is the
inner unit normal to $\partial\Omega$ at $\hat{x}$.
\end{lemma}

One of the obstacles in the present work is the lack of concavity of the operator $F$ appearing in (\ref{e2.5}).
It is important to note that one of the key elements of this  paper is to transform the equation (\ref{e1.3}) to find hidden concavity properties when $h(x)>(n-2)\frac{\pi}{2}$,  thus it allows us to obtain the gradient estimate.
It is worthy of mention that we use the argument following from \cite{CPW} and \cite{YY}.
\begin{lemma}\label{t3.2}
Let $F(\kappa)=\sum_{i=1}^{n}\arctan\kappa_{i}$ be defined on
\begin{equation*}
\{\kappa\in\mathbb{R}^{n}:\sum_{i=1}^{n}\arctan\kappa_{i}\geq (n-2)\frac{\pi}{2}+\delta\}.
 \end{equation*}
 Then there  exists  a sufficiently large $A$ depending only  on $\delta$ such that $G(\kappa)=-e^{-AF(\kappa)}$ is a concave function.
\end{lemma}

Similarly, an consequence of Lemma \ref{t3.2} is
\begin{lemma}\label{t3.3}
Let $\Gamma:=\{\mathcal{M}\in sym(n):F(\mathcal{M})\geq(n-2)\frac{\pi}{2}+\delta\}$, then there exists $A=A(\delta)$ such that the operator
\begin{equation*}
G(\mathcal{A})=-e^{-AF(\mathcal{A})}
\end{equation*}
is elliptic and concave on $\Gamma$.
\end{lemma}
The readers can see the proof of the above two lemmas in \cite{CPW}.
Now we use the concavity of the operator $G$ from Lemma 3.4. Without loss of generality, we assume that $\mathcal{A}$ is diagonal.
It is easy to verify that
\begin{equation*}
G^{ij}(\mathcal{A})=\frac{\partial G(\mathcal{A})}{\partial a_{ij}}=\frac{\partial G(\mathcal{A})}{\partial a_{ii}}\delta_{ij}, \quad G^{ij,kl}(\mathcal{A})=\frac{\partial^{2}G(\mathcal{A})}{\partial a_{ij}\partial a_{kl}},
\end{equation*}
then we have
\begin{equation}\label{e3.2}
G^{ij,kl}M_{ij}M_{kl}\leq 0,
\end{equation}
for any real symmetric tensor $M_{ij}$. A direct computation shows that there exists a  positive constant $\widetilde{c}$ such that
\begin{equation}\label{e3.3}
\sum _{i=1}^{n}\kappa_{i}\frac{\partial G(\kappa)}{\partial \kappa_{i}}=Ae^{-AF(\kappa)}\sum_{i}\frac{\kappa_{i}}{1+\kappa_{i}^{2}}\geq  -\widetilde{c},  \quad
 \mathrm{in}\quad \overline{\Gamma^{\sigma}}\setminus \{0\}.
\end{equation}
By the arguments  in \cite{CPW} one can conclude that there exists
  some constant $C_{2}$ such that
\begin{equation}\label{e3.4}
|\kappa|^{2}\sum_{i=1}^{n}\frac{\partial G(\kappa)}{\partial \kappa_{i}}\leq C_{2}\sum_{i=1}^{n}\frac{\partial G(\kappa)}{\partial \kappa_{i}}\kappa_{i}^{2},  \quad
 \mathrm{in}\quad \overline{\Gamma^{\sigma}}\setminus\{0\}.
\end{equation}
\par
After introducing the new operator $G$,  then (\ref{e1.3}) and (\ref{e1.4}) are equivalent to the following Dirichlet problem
 \begin{equation}\label{e3.5}
\left\{ \begin{aligned}
G(\mathcal{A})&=-e^{-Ah(x)}:= \psi(x),\ \ && \quad
 \mathrm{in}\quad \Omega, \\
u&=\varphi,  && \quad
 \mathrm{on}\quad\partial\Omega.
\end{aligned} \right.
\end{equation}
It can be followed from Lemma \ref{t2.3} that
 \begin{equation}\label{e3.6}
     \frac{\partial G(\kappa)}{\partial \kappa_{n}}=Ae^{-AF(\kappa)}\frac{1}{1+\kappa_{n}^{2}}\geq C_{3}>0.
   \end{equation}
where $C_{3}$ is a positive constant depending only on $\delta$.\par

For the convenient of reader, we denote
 $$g_{\alpha}=\frac{\partial G(\kappa)}{\partial \kappa_{\alpha}}(1\leq\alpha\leq n-1),\,\,g_{i}=\frac{\partial G(\kappa)}{\partial \kappa_{i}}(1\leq i\leq n),\,\,g_{n}=\frac{\partial G(\kappa)}{\partial \kappa_{n}}.$$

After the corresponding transformation, inspired by  literatures in \cite{CNS3} and \cite{CNS4}, then we derive that

\begin{lemma}\label{t3.4}
Suppose that $\Omega$, $\varphi$ and $h$ satisfy (A). If there exists a $C^{2}$  admissible subsolution $\underline{u}$ satisfying (\ref{e2.6}), $u\in C^{3}(\overline{\Omega})$ satisfy (\ref{e1.3}) and (\ref{e1.4}), then we have the following estimate
\begin{equation}\label{e3.7}
\sup_{\overline{\Omega}}|Du|\leq C_{4},
\end{equation}
where $C_{4}$ is a  positive constant that depends on $\overline{\Omega}$, $\| \underline{u}\|_{C^{2}(\overline{\Omega})}$, $\|h\|_{C^{2}(\overline{\Omega})}$, $\delta$.
\end{lemma}

\begin{proof}
Our first goal is to  estimate  $|Du|$ in $\Omega$. Let
\begin{equation*}
z=|Du|e^{\widetilde{A}u},
\end{equation*}
where $\widetilde{A}=\big[\frac{2}{C_{3}}\max|D\psi(x)|\big]^{1/2}$, $C_{3}$ is the  positive constant  from (\ref{e3.6}). We assume that it achieves its maximum at a point $x_{0}$ in $\Omega$. At this point we may suppose $|D u|=u_{n}>0$
and
$u_{\alpha}(x_{0})$=0 for $1\leq\alpha\leq n-1$, then $\log u_{n}+\widetilde{A}u$ reaches its maximum. It is obvious that
\begin{equation}\label{e3.8}
\frac{u_{ni}}{u_{n}}+\widetilde{A}u_{i}=0.
\end{equation}
Thus for $1\leq\alpha\leq n-1$ we have
\begin{equation}\label{e3.8.1}
u_{n\alpha}(x_{0})=-\widetilde{A}u_{\alpha}(x_{0})u_{n}(x_{0})=0, \quad u_{nn}(x_{0})=-\widetilde{A}u_{n}^{2}(x_{0}).
\end{equation}
By rotation, we may assume that $u_{ij}(x_{0})$ is diagonal. For any $i$ we also get
\begin{equation}\label{e3.9}
\Big(\frac{u_{nii}}{u_{n}}-\frac{u_{ni}^{2}}{u_{n}^{2}}-\widetilde{A}u_{ii}\Big)(x_{0})\leq0.
\end{equation}
From (\ref{t2.3}) we can find  $a_{ij}(x_{0})$ is also diagonal and
\begin{equation}\label{e3.10}
\begin{aligned}
a_{nn}&=\frac{u_{nn}}{w}\Big(1-\frac{2u_{n}^{2}}{w(1+w)}+\frac{u_{n}^{4}}{w^{2}(1+w)^{2}}\Big),\\
a_{\alpha \alpha}&=\frac{u_{\alpha \alpha}}{w},\quad 1\leq \alpha\leq n-1.
\end{aligned}
\end{equation}
Differentiating both sides of (\ref{e3.5}) with respect to $x_{n}$, one can deduced that
\begin{equation}\label{e3.11}
\psi_{n}=G^{ij}(\mathcal{A})\frac{\partial a_{ij}}{\partial x_{n}}.
\end{equation}
  It follows from (\ref{e3.10})  that
\begin{equation*}
\begin{aligned}
a_{\alpha \alpha,n}&=(\frac{1}{w})_{n}u_{\alpha \alpha}+\frac{1}{w}u_{\alpha \alpha n}=-\frac{u_{n}u_{nn}}{w^{2}}a_{\alpha \alpha}+\frac{u_{\alpha \alpha n}}{w},\\
a_{nn,n}&=(\frac{u_{nn}}{w})_{n}\Big(1-\frac{2u_{n}^{2}}{w(1+w)}+\frac{u_{n}^{4}}{w^{2}(1+w)^{2}}\Big)+\frac{u_{nn}}{w}\Big(1-\frac{2u_{n}^{2}}{w(1+w)}+\frac{u_{n}^{4}}{w^{2}(1+w)^{2}}\Big)_{n}\\
&=\frac{u_{nnn}}{w^{3}}-\frac{3u_{n}u_{nn}^{2}}{w^{5}}.
\end{aligned}
\end{equation*}
Thus we obtain
\begin{equation*}
\begin{aligned}
\psi_{n}&=\frac{1}{w^{3}}g_{n}u_{nnn}-\frac{3g_{n}u_{n}u_{nn}^{2}}{w^{5}}-\frac{u_{n}u_{nn}}{w^{2}}\sum_{1\leq\alpha\leq n-1}g_{\alpha}a_{\alpha\alpha}+\frac{1}{w}\sum_{1\leq\alpha\leq n-1}g_{\alpha}u_{n\alpha\alpha}\\
&=\frac{1}{w^{3}}g_{n}u_{nnn}-\frac{3g_{n}u_{n}u_{nn}^{2}}{w^{5}}+\frac{u_{n}u_{nn}}{w^{2}}g_{n}a_{nn}-\frac{u_{n}u_{nn}}{w^{2}}g_{n}a_{nn}\\
~~~&~~~~-\frac{u_{n}u_{nn}}{w^{2}}\sum_{1\leq\alpha\leq n-1}g_{\alpha}a_{\alpha\alpha}+\frac{1}{w}\sum_{1\leq\alpha\leq n-1}g_{\alpha}u_{n\alpha\alpha}\\
&=\frac{1}{w^{3}}g_{n}u_{nnn}-\frac{2g_{n}u_{n}u_{nn}^{2}}{w^{5}}-\frac{u_{n}u_{nn}}{w^{2}}g_{i}a_{ii}+\frac{1}{w}\sum_{1\leq\alpha\leq n-1}g_{\alpha}u_{n\alpha\alpha}.
\end{aligned}
\end{equation*}
As a result, we have
\begin{equation}\label{e3.12}
\frac{2g_{n}u_{n}u_{nn}^{2}}{w^{5}}+\frac{u_{n}u_{nn}}{w^{2}}g_{i}a_{ii}=-\psi_{n}+\frac{1}{w^{3}}g_{n}u_{nnn}+\frac{1}{w}\sum_{1\leq\alpha\leq n-1}g_{\alpha}u_{n\alpha\alpha}.
\end{equation}
By substituting (\ref{e3.9}) into (\ref{e3.12}), after some computation, we obtain
\begin{equation}\label{e3.13}
\frac{2g_{n}u_{n}u_{nn}^{2}}{w^{5}}+\frac{u_{n}u_{nn}}{w^{2}}g_{i}a_{ii}-\frac{g_{n}u_{nn}^{2}}{w^{3}u_{n}}+\widetilde{A}u_{n}g_{i}a_{ii}\leq-\psi_{n}.
\end{equation}
Arranging (\ref{e3.13}), we get
\begin{equation*}
\frac{g_{n}u_{nn}^{2}}{u_{n}w^{5}}(w^{2}-2)+\frac{u_{n}u_{nn}}{w^{2}}g_{i}a_{ii}+\widetilde{A}u_{n}g_{i}a_{ii}\leq-\psi_{n}.
\end{equation*}
Combining with (\ref{e3.8}),  we can get that
\begin{equation}\label{e3.14}
\frac{g_{n}\widetilde{A}^{2}u_{n}^{3}}{w^{5}}(w^{2}-2)+\frac{\widetilde{A}}{w^{2}}u_{n}g_{i}a_{ii}\leq-\psi_{n}.
\end{equation}
It's obvious that $w=\sqrt{1+u_{n}^{2}}$ at $x_{0}$.
 By Lemma \ref{t2.3}, here we noted that there exists a positive  constant $\tau>0$ such that $$\sum_{1\leq \alpha\leq n-1}g_{\alpha}a_{\alpha\alpha}=\sum_{1\leq \alpha\leq n-1}g_{\alpha}\kappa_{\alpha}=Ae^{-F(\kappa)}\sum_{1\leq \alpha\leq n-1}\frac{\kappa_{\alpha}}{1+\kappa_{\alpha}^{2}}\geq\tau>0,$$ therefore
\begin{equation*}
\frac{g_{n}\widetilde{A}^{2}u_{n}^{3}}{w^{5}}(u_{n}^{2}-1)+\frac{\widetilde{A}}{w^{2}}u_{n}g_{n}a_{nn} \leq\max|D\psi|.
\end{equation*}
From (\ref{e3.8.1}) and (\ref{e3.10}) we obtain
\begin{equation*}
\frac{g_{n}\widetilde{A}^{2}u_{n}^{3}}{w^{5}}(u_{n}^{2}-1)-\frac{g_{n}\widetilde{A}^{2}u_{n}^{3}}{w^{5}}=\frac{g_{n}\widetilde{A}^{2}u_{n}^{3}}{w^{5}}(u_{n}^{2}-2) \leq\max|D\psi|.
\end{equation*}
Using (\ref{e3.6}) we derive that
\begin{equation*}
\frac{u_{n}^{3}(u_{n}^{2}-2)}{w^{5}}=\frac{u_{n}^{3}(u_{n}^{2}-2)}{(1+u_{n}^{2})^{5/2}}\leq\frac{\max|D\psi|}{C_{3}\widetilde{A}^{2}}=\frac{1}{2}.
\end{equation*}
Hence we  conclude that $\sup_{\Omega}|Du|\leq C_{4}$.\par
Finally, we assume that $x_{0}\in\partial\Omega$ and $\nu$ is the
inner unit normal to $\partial\Omega$ at $x_{0}$ such that
$$D_{\nu}u(x_{0})=\sup_{\overline{\Omega}}|Du|.$$
We choose $\widetilde{p}>0$  such that line segment $\{x_{0}+\widetilde{p}\nu:0\leq \widetilde{p}\leq \widetilde{p}_{0}\}$ is contained in $\Omega$. Since $\underline{u}= u= \overline{u}$  on $\partial\Omega$, it follows from Lemma \ref{t3.1} that we obtain
\begin{equation*}
\frac{\underline{u}(x_{0}+\widetilde{p}\nu)-\underline{u}(x_{0})}{\widetilde{p}}\leq\frac{u(x_{0}+\widetilde{p}\nu)-u(x_{0})}{\widetilde{p}}\leq\frac{\overline{u}(x_{0}+\widetilde{p}\nu)-\overline{u}(x_{0})}{\widetilde{p}}.
\end{equation*}
Therefore $D_{\nu}\underline{u}(x_{0})\leq D_{\nu}u(x_{0})\leq D_{\nu}\overline{u}(x_{0})$.   Thus this completes the proof of Lemma \ref{t3.4}.
\end{proof}
From (\ref{e3.5}),
 we define the equation
\begin{equation}\label{e5.3}
\widetilde{G}(D^{2}u, Du)\triangleq G(\mathcal{A})=\psi(x).
\end{equation}
 We will use the notation
\begin{equation*}
\widetilde{G}_{ij}=\frac{\partial \widetilde{G}}{\partial r_{ij}},\quad \widetilde{G}_{i}=\frac{\partial \widetilde{G}}{\partial p_{i}},
\end{equation*}
where   $r$   represents for the second derivative  and $p$  represents for gradient variables.
By the property of the operator $G$ and the boundedness  of  $|Du|$, the operator $\widetilde{G}(\cdot,p)$ satisfies the structure conditions as same as Lemma 2.14 in \cite{GB}. Then another version of Lemma \ref{t2.5} in the following holds.
\begin{Corollary}\label{c3.5}
 Assume that $\Omega$, $\varphi$ and $h$ satisfy $(A)$, $\underline{u}$ is an admissible subsolution satisfying (\ref{e2.6}), $u\in C^{3,\alpha}(\bar{\Omega})$ is an admissible solution to (\ref{e1.3}) and (\ref{e1.4}). As before, let $\lambda(D^2 u)=(\lambda_1,\cdots, \lambda_n)$ are the eigenvalues of Hessian matrix $D^2 u$ according to $x$. Then there exsits $R_{0}$ depending only on $\underline{u}$ and $\delta$, such that for any $|\lambda|\geq R_{0}$, we have
\begin{equation}\label{e2.11}
\widetilde{G}_{ij}[\underline{u}_{ij}-u_{ij}]\geq\tau>0,
\end{equation}
where $\tau$ is a constant depending on only $\underline{u}$ and $\delta$.
\end{Corollary}

\section{Estimates for Second Derivatives from  Bounds on the Boundary}
In this section we will show that how to estimate the second derivatives of $u$ in $\Omega$ if we know bounds for them on $\partial\Omega$. Let us assume that
\begin{lemma}\label{t4.1}
Suppose that $\Omega$, $\varphi$ and $h$ satisfy (A). If there exists a $C^{4}$ admissible subsolution $\underline{u}$ of (\ref{e2.6}), $u\in C^{3}(\overline{\Omega})$ satisfy (\ref{e1.3}) and (\ref{e1.4}), then we have the following estimate
\begin{equation}\label{e4.1}
\sup_{\partial\Omega}|D^{2}u|\leq C_{5}.
\end{equation}
where $C_{5}$ is a constant depending on $\partial\Omega$, $\underline{u}$, $h$, $\delta$ and $n$.
\end{lemma}
 The proof of this statements will be deferred until the section 5. If (\ref{e4.1}) holds then we have
\begin{lemma}\label{t4.2}
Suppose that $\Omega$, $\varphi$ and $h$ satisfy (A). If there exists a $C^{4}$ admissible subsolution $\underline{u}$ satisfying (\ref{e2.6}), $u\in C^{3}(\overline{\Omega})$ satisfy (\ref{e1.3}) and (\ref{e1.4}), then we have the following estimate
\begin{equation}\label{e4.2}
\sup_{\overline{\Omega}} |D^{2}u|\leq C_{6},
\end{equation}
where $C_{6}$ is a constant depending on $\overline{\Omega}$, $\| \underline{u}\|_{C^{4}(\overline{\Omega})}$, $\|h\|_{C^{2}(\overline{\Omega})}$, $\delta$.
\end{lemma}
 We present useful knowledge for the proof of Lemma \ref{t4.2}.
\begin{lemma}\label{t4.2.1}
There exists a positive constant $\Lambda$  only depending on $\delta$ such that
$$\sum_{i}g_{i}=Ae^{-AF(\kappa)}\sum_{i}\frac{1}{1+\kappa_{i}^{2}}\geq\Lambda>0.$$
\end{lemma}
\begin{proof}
Using Lemma \ref{t2.3} and $\sup_{\overline{\Omega}}F<n\frac{\pi}{2}$, we have an estimate for the smallest eigenvalue
$$|\kappa_{n}|\leq C_{\delta},$$
where $C_{\delta}$ is  a positive constant depending on $\delta$.
Therefore there exist two positive constants $\Lambda_{1}$ and $\Lambda$  only depending on $\delta$ such that
$$\sum_{i}g_{i}=Ae^{-AF(\kappa)}\sum_{i}\frac{1}{1+\kappa_{i}^{2}}\geq \Lambda_{1}\frac{1}{1+\kappa_{n}^{2}}\geq \Lambda_{1}\frac{1}{1+ C_{\delta}^{2}}\geq \Lambda>0.$$
\end{proof}
In order to prove  Lemma \ref{t4.2}, we need to use the knowledge of \cite{CNS3} as follows.  From the proof of Lemma 3.4,  we have a bound for
\begin{equation}\label{e4.3}
C_{7}=2\max_{\overline{\Omega}} w,
\end{equation}
where $w=\sqrt{1+|Du|^{2}}$. Let
\begin{equation}\label{e4.4}
\widetilde{\theta}\triangleq\frac{1}{w},
\end{equation}
\begin{equation}\label{e4.5}
a=\frac{1}{k}\triangleq\frac{1}{2}\min_{\overline{\Omega}}\widetilde{\theta}.
\end{equation}
Obviously we have
\begin{equation}\label{e4.6}
\frac{1}{\widetilde{\theta}-a}\leq\frac{1}{a}=C_{7}.
\end{equation}
  For any $x^{0}\in \Omega$,  it is convenient to use the new coordinates by defining the surface by $\zeta(y)$, where $y$ are tangential coordinates to the surface at the point $(x^{0},u(x^{0}))$. We choose a local orthonormal frame $\{e_{i}\}$ for $1\leq i\leq n$ in neighborhood of $(x^{0},u(x^{0}))$ in $\zeta(y)$, and $e_{n+1}$ is the normal.   Introduce the new orthonormal vectors
\begin{equation*}
\varepsilon_{1},\cdots,\varepsilon_{n},\varepsilon_{n+1},
\end{equation*}
where $\varepsilon_{n+1}$ being the normal vector at $x^{0}$ with $\varepsilon_{n+1}=\frac{1}{w}(-u_{1},\cdots,-u_{n},1)$ and $\varepsilon_{1}$ corresponding to the tangential direction at $x^{0}$ with largest principal curvature. After proper rotation, we represent the surface near $(x^{0},u(x^{0}))$  by tangential coordinates $y_{1},\cdots,y_{n}$ and $\zeta(y)$ (summation is of the from 1 to $n$)
\begin{equation*}
x_{j}e_{j}+u(x)e_{n+1}=x_{j}^{0}e_{j}+u(x^{0})e_{n+1}+y_{j}\varepsilon_{j}+\zeta(y)\varepsilon_{n+1},
\end{equation*}
so $\nabla \zeta(0)=0$ and $x=x^{0}$ if and only if $y=0$. In this case we set that
$\Upsilon=\sqrt{1+|\nabla \zeta|^{2}}$. Following \cite{CNS5}, the first fundamental form of the graph of $\zeta(y)$ is
$$\widetilde{g}_{ij}=\delta_{ij}+\zeta_{i}\zeta_{j}. $$
The second fundamental form is given by
$$\widetilde{h}_{ij}=\frac{\zeta_{ij}}{\sqrt{1+|\nabla \zeta|^{2}}}.$$
Thus the normal curvature in the $\varepsilon_{1}$ direction is
\begin{equation*}
\kappa= \frac{\widetilde{h}_{11}}{\widetilde{g}_{11}}=\frac{\zeta_{11}}{(1+\zeta_{1}^{2})\Upsilon}.
\end{equation*}
In the $y$-coordinate we have the normal vector
\begin{equation}\label{e4.7}
N=-\frac{1}{\Upsilon}\zeta_{j}\varepsilon_{j}+\frac{1}{\Upsilon}\varepsilon_{n+1},
\end{equation}
\begin{equation}\label{e4.8}
\widetilde{\theta}=\frac{1}{w}=N\cdot e_{n+1}=\frac{1}{\Upsilon w}-\frac{1}{\Upsilon }\sum a_{j}\zeta_{j},
\end{equation}
where $a_{j}=\varepsilon_{j}\cdot e_{n+1}$, thus $\sum a_{j}^{2} \leq 1$. And  the equation (\ref{e3.5}) locally reads as
\begin{equation}\label{e4.8a}
G(\mathcal{A})=\widetilde{\psi}(y),
\end{equation}
where $\widetilde{\psi}(y)\triangleq \psi(x).$  By $\nabla \zeta(0)=0$, it follows from (\ref{e4.8}) that  at $y=0$  we have
\begin{equation}\label{e4.9}
\widetilde{\theta}_{i}=-a_{i}\zeta_{ii},\quad  i=1,\cdots,n.
\end{equation}
\begin{equation}\label{e4.10}
\widetilde{\theta}_{ii}=-a_{j}\zeta_{jii}-\frac{\zeta_{ii}^{2}}{w}, \quad i=1,\cdots,n.
\end{equation}
\begin{lemma}\label{t4.3}
 We claim that $\zeta_{1j}(x_{0})=0$ for $j>1$ at $x_{0}\in\partial\Omega$.
\end{lemma}
\begin{proof}
Let $e_{\theta}=e_{1}\cos \theta+e_{j}\sin\theta$. Then
 \begin{equation*}
\zeta_{e_{\theta}e_{\theta}}=\zeta_{11}\cos^{2}\theta+2\zeta_{1j}\cos\theta \sin\theta +\zeta _{jj}\sin^{2}\theta
\end{equation*}
has a maximum at $\theta=0$. It follows that
 \begin{equation*}
\frac{d }{d \theta}\big(\zeta_{e_{\theta}e_{\theta}}(x_{0})\big)\Big|_{\theta=0}=0.
\end{equation*}
This gives $\zeta_{1j}(x_{0})=0$.
\end{proof}

Using the knowledge of \cite{CNS3} as above, then we show the following\par

\noindent{\bf Proof of Lemma \ref{t4.2}.}

Set \begin{equation*}
M=\max_{x\in\overline{\Omega},1\leq i\leq n}\frac{1}{\widetilde{\theta}-a}\kappa_{i}(x),
\end{equation*}
 where  the maximum is taken over all principal curvatures $\kappa_{i}$.  We  assume that the maximum is attained at some point $x^{0}\in\Omega$ and $M>0$.  At this point corresponding to $y=0$, the function
\begin{equation*}
\frac{1}{\widetilde{\theta}-a}\frac{\zeta_{11}}{(1+\zeta_{1}^{2})\Upsilon}
\end{equation*}
takes its maximum  value equal to $M$ by the previous assumption. At the point  $y=0$, because $y_{1}$ direction is a direction of principal curvature, then by Lemma \ref{t4.3}, we get that $\zeta_{1j}=0$ for $j>1$. By rotating the $\varepsilon_{2},\cdots, \varepsilon_{n}$, without loss of generality we may have that $\zeta_{ij}(0)$ is diagonal. At $y=0$, the function $\ln\label{e4.12}\frac{1}{\widetilde{\theta}-a}\frac{\zeta_{11}}{(1+\zeta_{1}^{2})\Upsilon}$ reaches its maximum value, and taking the derivative twice in a row can be concluded as follows
\begin{equation}\label{e4.11}
\frac{\zeta_{11i}}{\zeta_{11}}-\frac{\widetilde{\theta}_{i}}{\widetilde{\theta}-a}-\frac{2\zeta_{1}\zeta_{1i}}{1+\zeta_{1}^{2}}-\frac{\Upsilon_{i}}{\Upsilon}=0,   \quad  i=1,\cdots n,
\end{equation}
\begin{equation}\label{e4.12}
0\geq\frac{\zeta_{11ii}}{\zeta_{11}}-\frac{\zeta_{11i}^{2}}{\zeta_{11}^{2}}-\Big(\frac{\widetilde{\theta}_{i}}{\widetilde{\theta}-a}\Big)_{i}-2\zeta_{1i}^{2}-\zeta_{ii}^{2},   \quad  i=1,\cdots n.
\end{equation}
Observing (\ref{e2.4}), at the origin we also see that $a_{ij}=\zeta_{ij}$ is diagonal and
\begin{equation}\label{e4.13}
\begin{aligned}
\frac{\partial a_{ij}}{\partial y_{k}}&=a_{ij,k}=\zeta_{ijk},\\
\frac{\partial^{2} a_{ij}}{\partial^{2} y_{1}}&=a_{ij,11}=\zeta_{ij11}-\zeta_{11}^{2}\Big(\zeta_{ij}+\delta_{i1}\zeta_{1j}+\delta_{j1}\zeta_{1i}\Big).
\end{aligned}
\end{equation}
It should be noted that  at a matrix $\mathcal{A}=[a_{ij}]$ which is diagonal,
\begin{equation}\label{e4.13a}
\frac{\partial G}{\partial a_{ij}}=\frac{\partial G}{\partial\kappa_{i}}\delta_{ij}=g_{i}\delta_{ij}.
\end{equation}
We proceed the differentiate the equation
  $$G(\mathcal{A})=\widetilde{\psi}(y),$$
then there holds
\begin{equation*}
\frac{\partial G}{\partial a_{ij}}\frac{\partial a_{ij}}{\partial y_{k}}=\widetilde{\psi}_{k}(y).
\end{equation*}
Using (\ref{e4.13}) and (\ref{e4.13a}), we find that at $y=0$, $$g_{i}\zeta_{iik}=\widetilde{\psi}_{k}$$
for all $k$.
Using the concavity of $G$ and (\ref{e4.13}), one can see that  at $y=0$,
 $$\widetilde{\psi}_{11}=\frac{\partial^{2} G}{\partial  a_{ij}\partial  a_{kl}}\frac{\partial a_{ij}}{\partial y_{1}}\frac{\partial a_{kl}}{\partial y_{1}}+\frac{\partial G}{\partial a_{ij}}a_{ij,11}\leq \frac{\partial G}{\partial a_{ij}}a_{ij,11}.$$   Thus  we have, at $y=0$,
\begin{equation}\label{e4.14a}
\begin{aligned}
\widetilde{\psi}_{11}
&\leq\frac{\partial G}{\partial a_{ij}}a_{ij,11} = g_{i}(\zeta_{ii11}-\zeta_{ii}\xi_{11}^{2})-2g_{1}\zeta_{11}^{3}\\
&\leq g_{i}\Big(\frac{\zeta_{i11}^{2}}{\zeta_{11}}+\zeta_{11}(\frac{\widetilde{\theta}_{i}}{\widetilde{\theta}-a})_{i}+2\zeta_{11}\zeta_{1i}^{2}+\zeta_{11}\zeta_{ii}^{2}-\zeta_{ii}\zeta_{11}^{2}\Big)-2g_{1}\zeta_{11}^{3}\\
&=\zeta_{11}g_{i}\Big(\frac{\zeta_{i11}^{2}}{\zeta_{11}^{2}}+(\frac{\widetilde{\theta}_{i}}{\widetilde{\theta}-a})_{i}+\zeta_{ii}^{2}-\zeta_{11}\zeta_{ii}\Big)\\
&=\zeta_{11}g_{i}\Big[(\frac{\widetilde{\theta}_{i}}{\widetilde{\theta}-a})^{2}+(\frac{\widetilde{\theta}_{i}}{\widetilde{\theta}-a})_{i}+\zeta_{ii}^{2}-\zeta_{11}\zeta_{ii}\Big]\\
&=\zeta_{11}g_{i}\Big[\frac{\widetilde{\theta}_{ii}}{\widetilde{\theta}-a}+\zeta_{ii}^{2}-\zeta_{11}\zeta_{ii}\Big]\\
&=\zeta_{11}g_{i}\Big[\frac{1}{\widetilde{\theta}-a}(-a_{j}\zeta_{jii}-\frac{\zeta_{ii}^{2}}{w})+\zeta_{ii}^{2}-\zeta_{11}\zeta_{ii}\Big]\\
&=-\zeta_{11}a_{j}g_{i}\frac{\zeta_{jii}}{\widetilde{\theta}-a}+\zeta_{11}g_{i}\Big[(1-\frac{1}{\widetilde{\theta}-a}\frac{1}{w})\zeta_{ii}^{2}-\zeta_{11}\zeta_{ii}\Big]\\
&=-\zeta_{11}a_{j}\widetilde{\psi}_{j}\frac{1}{\widetilde{\theta}-a}+\zeta_{11}g_{i}\Big[(1-\frac{\widetilde{\theta}}{\widetilde{\theta}-a})\zeta_{ii}^{2}-\zeta_{11}\zeta_{ii}\Big]\\
&=-\zeta_{11}a_{j}\widetilde{\psi}_{j}\frac{1}{\widetilde{\theta}-a}-\zeta_{11}g_{i}\zeta_{ii}^{2}\frac{a}{\widetilde{\theta}-a}-\zeta_{11}^{2}g_{i}\zeta_{ii}
.
\end{aligned}
\end{equation}
Here the first step follows from (\ref{e4.13}) while the  second step comes from (\ref{e4.12}).
Using (\ref{e4.14a}) we have
\begin{equation}\label{e4.14b}
\zeta_{11}g_{i}\zeta_{ii}^{2}\frac{a}{\widetilde{\theta}-a}+\zeta_{11}^{2}g_{i}\zeta_{ii}\leq-\zeta_{11}a_{j}\widetilde{\psi}_{j}\frac{1}{\widetilde{\theta}-a}-\widetilde{\psi}_{11}.
\end{equation}
Recall that   $\widetilde{\psi}(y)=-e^{-Ah(y)}$,
a routine computation gives rise to
$$\widetilde{\psi}_{11}=-A^{2}e^{-Ah(y)}(\frac{\partial h}{\partial y_{1}})^{2}+Ae^{-Ah(y)}\frac{\partial^{2} h}{\partial^{2} y_{1}}.$$
Since $\sum a_{j}^{2}<1$, we now choose a sufficiently  large   constant $A$ so that
$$-\zeta_{11}a_{j}\widetilde{\psi}_{j}\frac{1}{\widetilde{\theta}-a}-\widetilde{\psi}_{11}\leq C_{8}(1+\zeta_{11}),$$
where $C_{8}$ is a positive constant depending on $\overline{\Omega}$, $\widetilde{\psi}$, $\underline{u}$ and $\widetilde{\theta}$.
 Thus
 $$\zeta_{11}g_{i}\zeta_{ii}^{2}\frac{a}{\widetilde{\theta}-a}+\zeta_{11}^{2}g_{i}\zeta_{ii}\leq C_{8}(1+\zeta_{11}).$$
Since   (\ref{e4.14b}) and the definition of $M$,  we have
\begin{equation}\label{e4.14}
Ma\sum_{i} g_{i}\zeta_{ii}^{2}+M^{2}(\widetilde{\theta}-a)^{2}\sum_{i} g_{i}\zeta_{ii}\leq C_{9}(1+M).
\end{equation}
where $C_{9}$  is a positive constant depending only on the known data.
With the aid of (\ref{e3.3}), we can see that $\sum _{i}g_{i}\zeta_{ii}=\sum_{i}g_{i}\kappa_{i}\geq-\widetilde{c}$.
 It follows that from (\ref{e4.14}) that
 \begin{equation*}
Ma\sum_{i} g_{i}\zeta_{ii}^{2}\leq \widetilde{c}M^{2}(\widetilde{\theta}-a)^{2}+ C_{9}(1+M).
\end{equation*}
 Using (\ref{e3.4}) and the definition of $M$,  there exists a positive constant $C_{10}$ such that
\begin{equation*}
C_{10}M^{3}\sum_{i}g_{i}\leq \widetilde{c}M^{2}(\widetilde{\theta}-a)^{2}+ C_{9}(1+M).
\end{equation*}
By Lemma \ref{t4.2.1}, we obtain that $\sum_{i}g_{i}\geq\Lambda>0$, then
\begin{equation*}
M\leq C_{6},
\end{equation*}
where a suitable constant $C_{6}$  depends on $\overline{\Omega}$, $\| \underline{u}\|_{C^{4}(\overline{\Omega})}$, $\|h\|_{C^{2}(\overline{\Omega})}$, $\delta$. The proof is completed.

\section{Proof of Lemma 4.1}

The proof of Lemma \ref{t4.1} follows from the following steps.

First, we need to estimate the pure tangential second derivatives.  Suppose that the point is the origin and  the $x_{n}-$axis is the inner normal there. We may assume that the boundary near 0 is represented by
\begin{equation}\label{e5.1}
x_{n}=\rho(x')=\frac{1}{2}\sum_{\alpha,\beta< n}\rho_{\alpha\beta}(0)x_{\alpha}x_{\beta}+\mathrm{O}(|x'|^{3}),
\end{equation}
where $x'=(x_{1},x_{2},\cdots,x_{n-1})$. Using the Dirichlet boundary condition, we get
\begin{equation}\label{e5.2}
(u-\underline{u})(x',\rho(x'))=0.
\end{equation}
Taking two derivatives for  both sides of (\ref{e5.2}), we obtain
\begin{equation*}
(u-\underline{u})_{\alpha\beta}(0)=-(u-\underline{u})_{n}(0)\rho_{\alpha\beta}(0),\quad 1\leq\alpha,\beta\leq n-1.
\end{equation*}
From the boundary gradient estimate it follows that
\begin{equation*}
|u_{\alpha\beta}(0)|\leq C_{11},\quad \alpha,\beta<n.
\end{equation*}\par
The next thing to do in the proof is to estimate the mixed normal-tangential derivative $u_{\alpha n}(0)$ for $\alpha<n$.
 Using the pre-knowledge in the second section, we can obtain
\begin{equation}\label{e5.4}
\widetilde{G}_{ij}=G^{kl}\frac{\partial a_{kl}}{\partial r_{ij}}=\frac{1}{w}b^{ik}G^{kl}b^{lj},
\end{equation}
and
\begin{equation*}
\widetilde{G}_{i}=G^{kl}\frac{ \partial a_{kl}}{\partial p_{i}}=G^{kl}\frac{\partial}{\partial p_{i}}\Big( \frac{1}{w}b^{km}b^{ql}\Big)u_{mq}.
\end{equation*}
 Using a simple calculation, we have
\begin{equation}\label{e5.5}
\begin{aligned}
\widetilde{G}_{i}&=-\frac{u_{i}}{w^{2}}G^{kl}a_{kl}-\frac{2}{w}G^{kl}a_{lm}b^{ik}u_{m}\\
&=-\frac{u_{i}}{w^{2}}\sum _{j}g_{j}\kappa_{j}-\frac{2}{w}G^{kl}a_{lm}b^{ik}u_{m}.
\end{aligned}
\end{equation}
 Recalling (\ref{e3.7}), we see that $|Du|$ is bounded and the eigenvalues of $[b^{ij}]$ is bounded between two controlled positive constants.
Since (\ref{e5.5}), we see that in fact
\begin{equation}\label{e5.6a}
\sum_{i}|\widetilde{G}_{i}|\leq C_{12}\sum g_{i}|\kappa_{i}|.
\end{equation}
 Moreover, according to  Lemma \ref{t3.4}, (\ref{e2.3}) and (\ref{e5.4}), there exist two positive constants $C_{13}$, $C_{14}$ such that
\begin{equation}\label{e5.6}
C_{13}\sum_{i}G^{ii}\leq \sum_{i} \widetilde{G}_{ii}\leq C_{14}\sum_{i} G^{ii}.
\end{equation}
We also have
\begin{equation}\label{e5.7}
\sum_{i}\frac{\partial \widetilde{G}}{\partial \lambda_{i}}\lambda_{i}=\widetilde{G}_{ij}u_{ij}=\frac{1}{w}b^{ik}G^{kl}b^{lj}u_{ij}=G^{kl}a_{kl}=\sum g_{i}\kappa_{i}.
\end{equation}
where $(\lambda_{1},\cdots,\lambda_{n})$ are the eigenvalues of Hessian matrix $D^{2}u$ of $x\in\Omega$. From Lemma \ref{t2.3}, we have some  positive constants $s_{1}$ , $s_{2}$ and $s_{3}$ such that
\begin{equation}\label{e5.6ac}
0\leq \sum_{i}\frac{\partial G}{\partial \kappa_{i}}\leq s_{1},
\end{equation}
and
 $$s_{2}\leq \sum_{i}\frac{\partial G}{\partial \kappa_{i}}\kappa_{i}^{2}=Ae^{-AF(\kappa)}\sum_{i}\frac{\kappa_{i}^{2}}{1+\kappa_{i}^{2}}=Ae^{-AF(\kappa)}\sum_{i}\Big(1-\frac{1}{1+\kappa_{i}^{2}}\Big)\leq s_{3}.$$
 By taking   (\ref{e5.6}) and (\ref{e5.6ac}) into consideration, then there exists a positive constant $C_{15}$ such that
\begin{equation}\label{e5.8}
\sum\widetilde{G}_{ii}\leq C_{15}.
\end{equation}
\par
We consider  the operator
\begin{equation*}
L=\widetilde{G}_{ij}D_{i}D_{j}+\widetilde{G}_{i}D_{i}.
\end{equation*}
It is  necessary to obtain a version of differential inequality for the later construction of  auxiliary functions.\par
 By the argument as before in this section, one can conclude that for  any  point $x_{0}\in \partial\Omega$, we rotate the graph such that  at  $x_{0}$,  $D(u-\underline{u})=0$. But the principle curvatures are invariant through such transformation. Then for any sufficiently small positive constant $\epsilon>0$,  there exists $\delta'>0$ depending only on $\epsilon$,  $\Omega$, $\underline{u}$ and $h$ such that in $\Omega\cap B_{\delta'}(x_{0})$, we have
\begin{equation}\label{e5.5ab}
|D(u-\underline{u})|\leq\varepsilon.
\end{equation}
Based on the above argument, we obtain the following useful differential inequality.
\begin{Corollary}\label{t5.2}
There   exists  uniform  positive constant  $C_{16}$ such that for $|\lambda|\geq R_{0}$, we have
\begin{equation}\label{e5.5c}
L(u-\underline{u})\leq -C_{16}, \quad  \mathrm{in} \quad \Omega\cap B_{\delta'}(x_{0}).
\end{equation}
where $R_{0}$ is a  positive constant depending on $\tau$.
\end{Corollary}
\begin{proof}
Using (\ref{e5.5ab}) and (\ref{e5.6a}), we can see that  there exists  sufficiently small positive constant $\epsilon_{1}$
with $\epsilon_{1}\leq \frac{\tau}{2}$
such that in  $\Omega\cap B_{\delta'}(0)$, we have
\begin{equation}\label{e5.7ab}
\widetilde{G}_{i}(u-\underline{u})_{i}\leq\epsilon_{1}.
\end{equation}
From the Corollary \ref{c3.5} and (\ref{e5.7ab}), for $|\lambda|\geq R_{0}$ we have
\begin{equation}\label{e5.7a}
\begin{aligned}
L(u-\underline{u})&=\widetilde{G}_{ij}(u-\underline{u})_{ij}+\widetilde{G}_{i}(u-\underline{u})_{i}\\
&\leq-\tau+\epsilon_{1}\\
&\leq-C_{16}
\end{aligned}
\end{equation}
where $C_{16}=\frac{\tau}{2}$.
\end{proof}
\begin{lemma}\label{t5.1}
 There exist some uniform positive constants  $\delta'$ sufficiently small depending only on $\underline{u}, h, \Omega$ and choose positive constants $t_{1}, t_{2}, \epsilon_{2}$ depending on   $\delta'$
such that the function
\begin{equation*}
v=t_{1}(u-\underline{u})+t_{2}d-d^{2},
\end{equation*}
satisfy

\begin{equation*}
\begin{aligned}
Lv&\leq -\epsilon_{2},  &&\quad \mathrm{in} \quad \Omega\cap B_{\delta'}(0),\\
v&\geq0,   &&\quad \mathrm{on}\quad\partial(\Omega\cap B_{\delta'}(0)).
\end{aligned}
\end{equation*}
where $d(x)=d(x,\partial\Omega)$ is the distance function.
\end{lemma}
\begin{proof}
 Note that
 $v=t_{1}(u-\underline{u})+(t_{2}-d)d$ and let $t_{2}=2\delta'>d$, it is obvious that $v\geq0$ on $\partial(\Omega\cap B_{\delta'}(0))$. Since $$\widetilde{G}_{ij}D_{i}D_{j}d^{2}=2d\widetilde{G}_{ij}D_{i}D_{j}d+2\widetilde{G}_{ij}D_{i}dD_{j}d,\quad   \quad \mathrm{in} \quad \Omega\cap B_{\delta'}(0),$$
then we have
\begin{equation}\label{e5.10}
\begin{aligned}
Lv&=t_{1}\Big[\widetilde{G}_{ij}(u_{ij}-\underline{u}_{ij})+\widetilde{G}_{i}(u_{i}-\underline{u}_{i})\Big]+t_{2}\widetilde{G}_{ij}D_{i}D_{j}d+t_{2}\widetilde{G}_{i}D_{i}d\\
&~~-2d\widetilde{G}_{ij}D_{i}D_{j}d-2\widetilde{G}_{ij}D_{i}dD_{j}d-2d \widetilde{G}_{i}D_{i}d.\\
&=t_{1}\Big[\widetilde{G}_{ij}(u_{ij}-\underline{u}_{ij})+\widetilde{G}_{i}(u_{i}-\underline{u}_{i})\Big]+(t_{2}-2d)\widetilde{G}_{ij}D_{i}D_{j}d\\
&~~+(t_{2}-2d)\widetilde{G}_{i}D_{i}d-2\widetilde{G}_{ij}D_{i}dD_{j}d.
\end{aligned}
\end{equation}
Since $\partial\Omega$ is a $C^{4}$ bounded hypersurface,  we can assume that $d(x)$ is $C^{4}$ bounded in $\Omega\cap B_{\delta'}(0)$. It can be verified using a result of G-T (see \cite{GT} , Lemma 14.17) that  $d(x)$ satisfy
\begin{equation*}
\begin{aligned}
Dd(x)&=(0,\cdots,0,1),   \\ D^{2}d(x)&=diag(-\frac{\kappa_{1}}{1-\kappa_{1}d},\cdots,-\frac{\kappa_{n-1}}{1-\kappa_{n-1}d},0),
\end{aligned}
 \end{equation*}
 in $\Omega\cap B_{\delta'}(0)$,
 where  $\kappa_1,\cdots, \kappa_{n-1}$ are the principle curvature on $\partial\Omega$. Combining  the boundedness of  $\widetilde{G}_{i}$ and (\ref{e5.8}), we obtain
\begin{equation*}
\begin{aligned}
|\widetilde{G}_{ij}D_{i}D_{j}d|\leq C_{17}, \quad |\widetilde{G}_{i}D_{i}d|\leq C_{18},
\end{aligned}
 \end{equation*}
 where $C_{17}$ and $C_{18}$ are positive constants depending on $\Omega$, $\delta'$, $\delta$ and $h$.
There are two cases to consider: Firstly, we consider that $|\lambda|\leq R_{0}$, where $R_{0}$ is a constant from Corollary \ref{c3.5}.
Then there holds $|\kappa|\leq \tilde{R}_{0}$ for some
constant $\tilde{R}_{0}$ depending only on $R_{0}$ and the boundedness of $|Du|$.
Without loss of generality, we assume  that $[a_{ij}]$ is diagonal.  By making use of (\ref{e5.4}) and the boundedness of $|Du|$, it is also easy to see that
\begin{equation*}
\widetilde{G}_{ij}D_{i}dD_{j}d=\frac{1}{1+\kappa_{l}^{2}}\delta_{kl} \frac{1}{w}b^{ik}b^{jl}D_{i}dD_{j}d\geq\frac{C_{19}}{1+\tilde{R}_{0}^{2}}.
 \end{equation*}
For $|\lambda|\leq R_{0}$, we have
  \begin{equation*}
\begin{aligned}
|\widetilde{G}_{ij}(u_{ij}-\underline{u}_{ij})|&\leq C_{20} , \quad |\widetilde{G}_{i}(u_{i}-\underline{u}_{i})|\leq C_{21}.
\end{aligned}
 \end{equation*}
 It follows from (\ref{e5.10}) that
  \begin{equation*}
  \begin{aligned}
 Lv&\leq t_{1}(C_{20}+C_{21})+(t_{2}-2d)C_{17}+(t_{2}-2d)C_{18}-\frac{2C_{19}}{1+\tilde{R}_{0}^{2}}\\
 & \leq t_{1}C_{22}+(t_{2}-2d)C_{23}-\frac{2C_{19}}{1+\tilde{R}_{0}^{2}}\\
 & \leq t_{1}C_{22}+t_{2}C_{23}-\frac{2C_{19}}{1+\tilde{R}_{0}^{2}}.
\end{aligned}
 \end{equation*}
 Secondly, if  $|\lambda|\geq R_{0}$,
    by using Corollary \ref{t5.2} we can obtain
  \begin{equation*}
  \begin{aligned}
 Lv &\leq -C_{16} t_{1}+(t_{2}-2d)C_{17}+(t_{2}-2d)C_{18}-2\widetilde{G}_{ij}D_{i}dD_{j}d\\
 &\leq-C_{16} t_{1}+(t_{2}-2d)C_{23}-2\widetilde{G}_{ij}D_{i}dD_{j}d\\
 &\leq-C_{16} t_{1}+t_{2}C_{23}.
 \end{aligned}
 \end{equation*}
 We may fix $t_{1}$ and $t_{2}$ such that $C_{16} t_{1}=2t_{2}C_{23}$.
 Then for $|\lambda|\leq R_{0}$, we have
  \begin{equation*}
  \begin{aligned}
 Lv&\leq t_{2}C_{24}-\frac{2C_{19}}{1+\tilde{R}_{0}^{2}}=2\delta'-\frac{2C_{19}}{1+\tilde{R}_{0}^{2}},
\end{aligned}
 \end{equation*}
 and for $|\lambda|\geq R_{0}$, we get
 \begin{equation*}
  \begin{aligned}
 Lv &\leq-t_{2}C_{23}=-2\delta'C_{23}.
 \end{aligned}
 \end{equation*}
 Let
 $$\delta'\leq \frac{C_{19}}{2(1+\tilde{R}_{0}^{2})}.$$
 Then the proof is completed. \end{proof}

By  making use of Lemma \ref{t2.3} (i), we find that it can't   be all negative for $\kappa$. Therefore, we can present
\begin{rem}\label{t5.2a}
By the assumption,  $\lambda=(\lambda_{1},\lambda_{2},\cdots,\lambda_{n})$ satisfy $ \lambda_{1}\geq \lambda_{2}\geq \cdots \geq \lambda_{n}$
and $ \lambda_{1}\geq \lambda_{2}\geq \cdots \geq \lambda_{n-1}\geq 0$.
We claim that  $ \lambda_{n}< -R_{0}$ is not satisfied.
\end{rem}

In order to estimate $|u_{\alpha n}(0)|$, we   introduce the motivation for the vector field $T_{\alpha}$ by
\begin{equation*}
T_{\alpha}=\frac{\partial}{\partial x_{\alpha}}+\sum_{\beta< n}\rho_{\alpha\beta}(0)\big(x_{\beta}\frac{\partial}{\partial x_{n}}-x_{n}\frac{\partial}{\partial x_{\beta}}\big),\quad \alpha\in\{1,2,\cdots,k-1\}.
 \end{equation*}
Furthermore, it follows immediately from (\ref{e5.1}) that for $\alpha<n$, on $\partial\Omega$ near $0$ we have see
 \begin{equation*}
T_{\alpha}=[\frac{\partial}{\partial x_{\alpha}}+\frac{\partial \rho}{\partial x_{\alpha}}\frac{\partial}{\partial x_{n}}]+\mathrm{O}(|x'|^{2})\frac{\partial}{\partial x_{n}}-\sum_{\beta< n}\rho_{\alpha\beta}(0)\rho(x')\frac{\partial}{\partial x_{\beta}}.
 \end{equation*}
 Now, we can show that
 \begin{lemma} \label{t5.3}
 The function $u-\underline{u}$ satisfies the following estimates
 \begin{equation}\label{e5.11}
|T_{\alpha}(u-\underline{u})|\leq C_{24}|x'|^{2},\quad \mathrm{on}\quad\partial (\Omega\cap B_{\delta'}(0)).
\end{equation}
 \begin{equation}\label{e5.11}
|LT_{\alpha}(u-\underline{u})|\leq  C_{25}(1+\sum_{i=1}^{n}\widetilde{G}_{ii})\leq C_{26}, \quad \mathrm{in} \quad\Omega\cap B_{\delta'}(0).
  \end{equation}

 \end{lemma}
 \begin{proof}
 Since $u=\underline{u}$ on $\partial \Omega$, then the first inequality follows directly from $C^{1}$ estimate of Lemma \ref{t3.4}. In the following we prove the last inequality. Since $G(\kappa)$ is invariant under rotation, it follows that for  the operator  $$X_{\beta}=x_{\beta}\frac{\partial}{\partial x_{n}}-x_{n}\frac{\partial}{\partial x_{\beta}},$$ which is the infinitesimal generator of a rotation. We can apply  the operator $X_{\beta}$ to the equation $\widetilde{G}(D^{2}u, Du)=\psi$. Consequently we have $$X_{\beta}\psi=\widetilde{G}_{ij}D_{i}D_{j}(X_{\beta} u)+\widetilde{G}_{i}D_{i}(X_{\beta} u),$$ from which it follows that $$T_{\alpha}\psi=L(T_{\alpha}u).$$ Since  $\widetilde{G}_{ii}$ is bounded, we have
\begin{equation*}
 \begin{aligned}
|LT_{\alpha}(u-\underline{u})|&\leq|LT_{\alpha}u|+|LT_{\alpha}\underline{u}|\\
&=|T_{\alpha}\psi|+|LT_{\alpha}\underline{u}|\\
&\leq C_{25}\Big(1+\sum_{i}\widetilde{G}_{ii}\Big)\\
&\leq C_{26},\quad \mathrm{in}\quad \Omega\cap B_{\delta'}(0).
\end{aligned}
\end{equation*}
 \end{proof}
 The function $\wp$ introduced in the next Lemma will be the main part of a barrier function which we shall construct in the following.
 \begin{lemma}\label{t5.4}
 There exist two positive constants $\widehat{A}\gg \widehat{B}\gg 1$ such that
 \begin{equation*}
 \wp=\widehat{A}v+\widehat{B}|x|^{2}\pm T_{\alpha}(u-\underline{u})
 \end{equation*}
 satisfies the inequalities
  \begin{equation*}
\begin{aligned}
L\wp&\leq 0,\quad \mathrm{in}\quad \Omega\cap B_{\delta'}(0),\\
\wp&\geq0,\quad \mathrm{on}\quad \partial(\Omega\cap B_{\delta'}(0)),
\end{aligned}
\end{equation*}
where $\delta'$ is a constant from Corollary \ref{t5.2}.
 \end{lemma}
 \begin{proof}
 First the condition $\wp\geq0$ on $\partial(\Omega\cap B_{\delta'}(0))$ follows if
 \begin{equation*}
 \widehat{B}|x|^{2}\pm T_{\alpha}(u-\underline{u})\geq0,\quad \mathrm{on}\quad\partial(\Omega\cap B_{\delta'}(0)).
 \end{equation*}
 In view of Lemma \ref{t5.3}, this can be arranged by choosing $\widehat{B}$ sufficiently large. By making use of  Lemma \ref{t5.1}  and Lemma  \ref{t5.3},  we can see that the property $L\wp\leq0$ now follows from the inequality
 \begin{equation*}
 \begin{aligned}
 L\wp&=\widehat{A}Lv+\widehat{B}L(|x|^{2})\pm LT_{\alpha} (u-\underline{u})\\
 &\leq- \widehat{A}\varepsilon_{2}+\widehat{B}(2\widetilde{G}_{ii}+2\widetilde{G}_{i}x_{i})+ C_{26}\\
 &\leq- \widehat{A}\varepsilon_{2}+\widehat{B}C_{27}+ C_{26}\\
 \end{aligned}
 \end{equation*}
 with holds if $\widehat{A}$ is  sufficiently large.
 \end{proof}
 The maximum principle applied to Lemma \ref{t5.4} shows that $\wp\geq0$ in $\Omega\cap B_{\delta'}(0)$. Since $\Big(\widehat{A}v+\widehat{B}|x|^{2}\pm T_{\alpha}(u-\underline{u})\Big)(0)=0$, it follows that
  \begin{equation*}
\frac{\partial}{\partial x_{n}}\big(\widehat{A}v+\widehat{B}|x|^{2}\pm T_{\alpha}(u-\underline{u})\big)(0)\geq 0.
  \end{equation*}
  Thus we obtain
  \begin{equation*}
  |\frac{\partial}{\partial x_{n}}(T_{\alpha}(u-\underline{u}))(0)|\leq|\frac{\partial}{\partial x_{n}}(\widehat{A}v+\widehat{B}|x|^{2})(0)|,
 \end{equation*}
 and this  gives
\begin{equation*}
|u_{\alpha n}(0)-\underline{u}_{\alpha n}(0)|\leq |\widehat{A}u_{n}(0)|+|\sum_{\beta< n}\rho_{\alpha\beta}(0)(u-\underline{u})_{\beta}(0)|\leq C_{28}.
  \end{equation*}
 Hence we can obtain the mixed second-order derivative bound $|u_{\alpha n}(0)|\leq C_{29}$ for all $\alpha<n$.\par
Finally, estimating the remaining second derivative $|u_{nn}|$ of the function $u$ is somewhat complicated.  We find that the  principal curvatures admit a uniform
 lower bound, i.e. $\mathcal{A}=[a_{ij}]\geq- C_{30}$, since the Lemma \ref{t2.3} and $F(\kappa)=\sum_{i}\arctan \kappa_{i}\geq(n-2)\frac{\pi}{2}+\delta$ in $\overline{\Omega}$.  Recall that $a_{ij}=\frac{1}{w}b^{ik}u_{kl}b^{lj}$ and the gradient estimate, we can obtain that  $D^{2}u\geq-C_{31}$. It suffices to
derive an upper bound
\begin{equation*}
u_{nn} \leq C_{32}, \quad \mathrm{on} \quad \partial\Omega.
\end{equation*}
 For discussion purposes, we need to give some definition of   the operator  as follows. Assume that the  $(n-1)\times(n-1)$  unper block matrix $u_{\alpha\beta}$ is diagonal. Before we deal with $D^{2}u$,  using the boundedness  of $Du$,
we consider $\widetilde{G}$ and  $\widehat{G}$ as functions through taking  $D^{2}u$ as variable all alone.   In other word, we can define that
$$\widetilde{G}[u_{ij}]:=\widetilde{G}(D^{2}u,p),$$
and
$$\widehat{G}[u_{\alpha\beta}]=\lim_{\mathrm{t}\rightarrow\infty }\widetilde{G}(D^{2}u+\mathrm{t}q\otimes q, p),$$
where  $q=(0,\ldots, 0,1)$ and $p$  represents for gradient variables.\par

Similar to Lemma \ref{t3.1}, we can deduce that
\begin{Corollary}\label{t5.5}
The operator $\widehat{G}[u_{\alpha\beta}]$ is concave on the set $u_{\alpha\beta}(\partial\Omega)$.
\end{Corollary}
 We need to
 consider
$$\Theta =\min _{x \in \partial\Omega }\widehat{G}[u_{\alpha\beta}(x)]-\psi(x)$$
 and
$$\mathbf{c}=\min_{x\in\partial\Omega}(\widehat{G}[\underline{u}_{\alpha\beta}]-\widetilde{G}[\underline{u}_{ij}])$$
that demonstrate $u_{nn} \leq C_{32}$ on $ \partial\Omega$.
It's obvious that $ \mathbf{c}>0$.
 According to an idea of Trudinger \cite{TN},
 we  use the following Lemma that implies the boundary $C^{2}$ estimate.
   \begin{lemma}\label{t5.6}
   There exist some positive constant $\omega_{0}>0$  depending only on the known data, such that  \begin{equation}\label{e5.14a}
\Theta\geq\omega_{0}>0.
\end{equation}
   \end{lemma}
  \begin{proof}
 Choose proper coordinates in $\mathbb{R}^{n}$ such that $\Theta$ is achieved at $0\in \partial\Omega$. By the Corollary \ref{t5.5}, there exists a symmetric $\{\widehat{G}_{0}^{\alpha\beta}\}$ such that \begin{equation}\label{e5.15}
 \widehat{G}_{0}^{\alpha\beta}\Big(u_{\alpha\beta}(x)-u_{\alpha\beta}(0)\Big)\geq \widehat{G}[u_{\alpha\beta}(x)]-\widehat{G}[u_{\alpha\beta}(0)],
 \end{equation}
where $ \widehat{G}_{0}^{\alpha\beta}=\frac{\partial \widehat{G}}{\partial u_{\alpha\beta}}(u_{\alpha\beta}(0))$.
Since (\ref{e5.1}) and (\ref{e5.2}),
by calculating we obtain
$$(u-\underline{u})_{\alpha}+(u-\underline{u})_{n}\rho_{\alpha}=0,$$
$$(u-\underline{u})_{\alpha \beta}+(u-\underline{u})_{\alpha n}\rho_{\beta}+(u-\underline{u})_{n\beta}\rho_{\alpha}+(u-\underline{u})_{n}\rho_{\alpha\beta}=0.$$
Since $\rho_{\alpha}=\rho_{\beta}=0$ on $\partial\Omega$, we have
\begin{equation}\label{e5.16}
u_{\alpha\beta}-\underline{u}_{\alpha\beta}=-(u-\underline{u})_{n}\rho_{\alpha\beta},   \quad \mathrm{on} \quad  \partial\Omega. \end{equation}
It follows that
\begin{equation}\label{e5.17}
\begin{aligned}
(u-\underline{u})_{n}(0)\widehat{G}_{0}^{\alpha\beta}\rho_{\alpha\beta}(0)&=\widehat{G}_{0}^{\alpha\beta}\Big(\underline{u}_{\alpha\beta}(0)-u_{\alpha\beta}(0)\Big)\\
&\geq\widehat{G}[\underline{u}_{\alpha\beta}(0)]-\widehat{G}[u_{\alpha\beta}(0)]\\
&=\widehat{G}[\underline{u}_{\alpha\beta}(0)]-\psi(0)-\Theta\\
&\geq\mathbf{c}-\Theta.
\end{aligned}
\end{equation}
Suppose  now that
$$(u-\underline{u})_{n}(0)\widehat{G}_{0}^{\alpha\beta}\rho_{\alpha\beta}(0)> \frac{1}{2}\mathbf{c}.$$
 Otherwise, taking (\ref{e5.17}) into consideration, we have $\Theta\geq \frac{1}{2}\mathbf{c}$. Hence we see that (\ref{e5.14a}) holds. \par
 Let $\eta:=\widehat{G}_{0}^{\alpha\beta}\rho_{\alpha\beta}$. By   Lemma \ref{t3.1a} we obtain  $(u-\underline{u})_{n}>0$. Taking (\ref{e3.7}) and $0< (u-\underline{u})_{n}\leq C_{\iota}$ into consideration,   there exists a universal constant $c_{\tau}>0$ depending on the $C^{1}$ estimate  and $\underline{u}$ such that
\begin{equation}\label{e5.18}
\eta(0)\geq\frac{1}{2}\mathbf{c}\frac{1}{(u-\underline{u})_{n}(0)}\geq 2 c_{\tau}>0,
\end{equation}
where$$c_{\iota}=\frac{1}{4}\frac{\mathbf{c}}{C_{\iota}}.$$
So we may assume that $\eta\geq c_{\tau} $ on $\Omega \cap B_{\delta'}(0)$ by  choosing $\delta'$ as in Lemma \ref{t5.1}.\par
Consider the function
$$\Phi=-(u-\underline{u})_{n}+\frac{1}{\eta}\widehat{G}_{0}^{\alpha\beta}\Big(\underline{u}_{\alpha\beta}(x)-u_{\alpha\beta}(0)\Big)-\frac{\psi(x)-\psi(0)}{\eta}, \quad \mathrm{in} \quad \Omega \cap B_{\delta'}(0).$$
We have $\Phi(0)=0$.  Using  (\ref{e5.15}) and (\ref{e5.16}),  we see that
\begin{equation}\label{e5.19}
\begin{aligned}
\Phi
&=\frac{1}{\eta}\Big\{-(u-\underline{u})_{n}\widehat{G}_{0}^{\alpha\beta}\rho_{\alpha\beta}+ \widehat{G}_{0}^{\alpha\beta}\big(\underline{u}_{\alpha\beta}(x)-u_{\alpha\beta}(0))\Big\}-\frac{1}{\eta}\Big(\psi(x)-\psi(0)\Big)\\
&=\frac{1}{\eta}\Big\{\widehat{G}_{0}^{\alpha\beta}(u_{\alpha\beta}(x)-\underline{u}_{\alpha\beta}(x)+\underline{u}_{\alpha\beta}(x)-u_{\alpha\beta}(0)\big)\Big\}-\frac{1}{\eta}\Big(\psi(x)-\psi(0)\Big)\\
&=\frac{1}{\eta}\Big\{\widehat{G}_{0}^{\alpha\beta}(u_{\alpha\beta}(x)-u_{\alpha\beta}(0))\Big\}-\frac{1}{\eta}\Big(\psi(x)-\psi(0)\Big)\\
&\geq \frac{1}{\eta}\Big\{\widehat{G}[u_{\alpha\beta}(x)]-\widehat{G}[u_{\alpha\beta}(0)]-(\psi(x)-\psi(0))\Big\}\\
&=\frac{1}{\eta}\Big\{\widehat{G}[u_{\alpha\beta}(x)]-\psi(x)-\Theta\Big\}\\
&\geq0, \quad \mathrm{on}\quad \partial\Omega.
\end{aligned}
\end{equation}
After direct computation, we have
\begin{equation*}
\begin{aligned}
D_{i}\Big[\frac{1}{\eta}\widehat{G}_{0}^{\alpha\beta}(\underline{u}_{\alpha\beta}(x)-u_{\alpha\beta}(0))\Big]&=\frac{1}{\eta}\widehat{G}_{0}^{\alpha\beta}\underline{u}_{\alpha\beta i}-\frac{\eta_{i}}{\eta^{2}}\widehat{G}_{0}^{\alpha\beta}(\underline{u}_{\alpha\beta}(x)-u_{\alpha\beta}(0)),\\
D_{i}D_{j}\Big[\frac{1}{\eta}\widehat{G}_{0}^{\alpha\beta}\big(\underline{u}_{\alpha\beta}(x)-u_{\alpha\beta}(0)\big)\Big]&=\frac{1}{\eta}\widehat{G}_{0}^{\alpha\beta}\underline{u}_{\alpha\beta ij}+D_{i}D_{j}(\frac{1}{\eta})\widehat{G}_{0}^{\alpha\beta}(\underline{u}_{\alpha\beta}(x)-u_{\alpha\beta}(0))-\frac{2}{\eta^{2}}\widehat{G}_{0}^{\alpha\beta}\underline{u}_{\alpha\beta i}\eta_{j}.
\end{aligned}
\end{equation*}
 Since $\widetilde{G}_{i}$ is  bounded and  $\psi$ is a bounded function, we can obtain that
\begin{equation*}
\begin{aligned}
L(\Phi)&=\widetilde{G}_{ij}D_{i}D_{j}\Phi+\widetilde{G}_{i}D_{i}\Phi\\
&=-\widetilde{G}_{ij}D_{i}D_{j}(u-\underline{u})_{n}-\widetilde{G}_{i}D_{i}(u-\underline{u})_{n}+\widetilde{G}_{ij}D_{i}D_{j}\Big[\frac{1}{\eta}\widehat{G}_{0}^{\alpha\beta}\big(\underline{u}_{\alpha\beta}(x)-u_{\alpha\beta}(0)\big)\Big]\\
~~&~+\widetilde{G}_{i}D_{i}\Big[\frac{1}{\eta}\widehat{G}_{0}^{\alpha\beta}(\underline{u}_{\alpha\beta}(x)-u_{\alpha\beta}(0))\Big]+\widetilde{G}_{ij}D_{i}D_{j}\frac{\psi(x)-\psi(0)}{\eta}+\widetilde{G}_{i}D_{i}\frac{\psi(x)-\psi(0)}{\eta}\\
&=-\psi_{n}+\widetilde{G}_{ij}\underline{u}_{nij}+\widetilde{G}_{i}\underline{u}_{ni}+\frac{1}{\eta}\widetilde{G}_{ij}\widehat{G}_{0}^{\alpha\beta}\underline{u}_{\alpha\beta ij}+D_{i}D_{j}(\frac{1}{\eta})\widetilde{G}_{ij}\widehat{G}_{0}^{\alpha\beta}(\underline{u}_{\alpha\beta}(x)-u_{\alpha\beta}(0))\\
~~&~-\frac{2}{\eta^{2}}\widetilde{G}_{ij}\widehat{G}_{0}^{\alpha\beta}\underline{u}_{\alpha\beta i}\eta_{j}+\frac{1}{\eta}\widetilde{G}_{i}\widehat{G}_{0}^{\alpha\beta}\underline{u}_{\alpha\beta i}-\frac{\eta_{i}}{\eta^{2}}\widetilde{G}_{i}\widehat{G}_{0}^{\alpha\beta}(\underline{u}_{\alpha\beta}(x)-u_{\alpha\beta}(0))\\
~~&~+\widetilde{G}_{ij}D_{i}D_{j}\frac{\psi(x)-\psi(0)}{\eta}+\widetilde{G}_{i}D_{i}\frac{\psi(x)-\psi(0)}{\eta}\\
&\leq C_{33}(1+\sum_{i}\widetilde{G}_{ii})\\
&\leq C_{34}(1+\sum_{i}G^{ii}).
\end{aligned}
\end{equation*}
 We can again use Lemma \ref{t5.1}, Lemma \ref{t5.4} and take $\widehat{A}\gg\widehat{B}\gg1$ large enough to get
\begin{equation}\label{e5.20}
\begin{aligned}
L(\widehat{A}v+\widehat{B}|x|^{2}+\Phi)&\leq0,\quad \mathrm{in}\quad \Omega\cap B_{\delta'}(0),\\
\widehat{A}v+\widehat{B}|x|^{2}+\Phi&\geq 0,\quad  \mathrm{on}\quad\partial(\Omega\cap B_{\delta'}(0)).
\end{aligned}
\end{equation}
 By the maximum principle, we see that $$\widehat{A}v+\widehat{B}|x|^{2}+\Phi\geq 0, \quad \mathrm{in}\quad \Omega\cap B_{\delta'}(0).$$We have
$$\widehat{A}v_{n}(0)+\Phi_{n}(0) = \frac{\partial}{\partial x_{n}}(\widehat{A}v + \widehat{B}|x|^{2}+\Phi)(0)\geq0,$$
since $\widehat{A}v+\widehat{B}|x|^{2}+\Phi= 0$ at the origin. Thus one can deduce that $\Phi_{n}(0)\geq -\widehat{A}v_{n}(0)\geq- C_{35}$.
Consequently we can conclude that $$u_{nn}(0)\leq C_{32}.$$ \par
So that we obtain  $|D^{2}u(0)|\leq C_{36}$ for some positive constant.
By letting $\mathrm{t}$ large enough, it yields that $[D^{2}u+\mathrm{t}q\otimes q]>[D^{2}u]$ uniformly.
Thus we see that $$\widehat{G}[u_{\alpha\beta}(0)]>\psi(0).$$
As a result, (\ref{e5.14a}) holds.
 \end{proof}
Now  we need   the following Lemma from \cite{CNS3}.
\begin{lemma}\label{t5.7}
Consider the $n\times n$ symmetric matrix
\begin{equation*}
M=\left(
\begin{matrix}
d_{1}&0&\cdots&0&a_{1}\\
0&d_{2}&\cdots&0&a_{2}\\
\vdots&\vdots&\ddots&\vdots&\vdots\\
0&0&\cdots&d_{n-1}&a_{n-1}\\
a_{1}&a_{2}&\cdots&a_{n-1}&a
\end{matrix}\right)
\end{equation*}
 with $d_{1}, \ldots, d_{n-1}$ fixed,  $a$ tending to infinity and $|a_{\alpha}|\leq C_{37}$ for $1\leq  \alpha\leq n-1$. Then the eigenvalues $\lambda_{i}$ of $M$ behave like
\begin{equation*}
\begin{aligned}
\lambda_{\alpha}&=d_{\alpha}+o(1),\quad 1\leq\alpha\leq n-1,\\
\lambda_{n}&=a\Big(1+ O(\frac{1}{a})\Big),
\end{aligned}
\end{equation*}
 for $1\leq i\leq n$, where the $o(1)$ and $O(\frac{1}{a})$ are uniform--depending only on  $d_{1}, \ldots, d_{n-1}$ and $C_{37}$.
\end{lemma} \par

Therefore, combining  Lemma \ref{t5.6}  with Lemma \ref{t5.7} we obtain that
   \begin{lemma} \label{t5.8}
   There exist suitable constants $R_{0},\omega_{0}>0$ such that for all $\mathrm{t}\geq R_{0}$ we have
   \begin{equation*}
  G(\kappa'(a_{\alpha\beta}),\mathrm{t})>\psi(x)+\omega_{0}.
   \end{equation*}
  where $\kappa'(a_{\alpha\beta})$ is the eigenvalue of the  $(n-1)\times(n-1)$ matrix $a_{\alpha\beta}$.
   \end{lemma}
   \begin{proof}
  Since (\ref{e5.14a}) holds, we have
 $$\lim_{t\rightarrow\infty }\widetilde{G}(r+\mathrm{t}q\otimes q, p)\geq \psi(x)+\omega_{0}.$$
 Writing $\mathcal{A}(r, p)=[a_{ij}]$  and $\mathcal{A}(r+\mathrm{t}q\otimes q, p)=[\widetilde{a}_{ij}]$, we have
 $$\widetilde{a}_{ij}=a_{ij}+\mathrm{t}b^{in}b^{nj},$$
 where $b^{ij}=\delta_{ij}-\frac{D_{i}u D_{j}u}{w(1+w)}$.
 After an orthonormal transformation we may assume
 $$\widetilde{a}_{ij}=a_{ij}+\frac{\mathrm{t}}{\sqrt{1+|Du|^{2}}}\delta_{in}\delta_{nj}.$$
 By Lemma \ref{t5.7},  the eigenvalues of $\mathcal{A}(r+\mathrm{t}q\otimes q, p)=[\widetilde{a}_{ij}]$ are given by
 $$\kappa'_{\alpha}=\kappa_{\alpha}+o(1), \quad  \tilde{\kappa}_{n}=\mathrm{t}+\kappa_{n}+O(1),$$
 as $\mathrm{t}\rightarrow\infty $, where $\kappa'_{\alpha}$ are the eigenvalues of $[\tilde{a}_{\alpha\beta}]_{1\leq\alpha,\beta\leq n-1}$.
  Through our observation,  we have $$\widetilde{G}(r+\mathrm{t}q\otimes q, p)=G(\kappa(\widetilde{a}_{ij}))=G(\kappa_{\alpha}+o(1),\mathrm{t}+\kappa_{n}+O(1)),$$
 where $\kappa_{n}\geq -\tan\delta$. Then there holds
 $$\lim_{t\rightarrow+\infty}G(\kappa'(a_{\alpha\beta}),\mathrm{t})=\hat{G}(u_{\alpha\beta}).$$
 Applying Lemma \ref{t5.6}, if $\mathrm{t}\rightarrow\infty $, we obtain
  $$ G(\kappa'(a_{\alpha\beta}),\mathrm{t})>\psi(x)+\omega_{0}.$$
   \end{proof}

Recall that
 \begin{equation*}
 G(\kappa)=-e^{-A\sum\arctan \kappa_{i}},\quad \psi(x)=-e^{-Ah(x)}.
 \end{equation*}
 Suppose that $\kappa_{1},\ldots,\kappa_{n}$ are the eigenvalues of $a_{ij}(0)$ for $x\in\partial\Omega$. Applying Lemma \ref{t5.7}, we see that the eigenvalues $\kappa_{1},\kappa_{2},\ldots,\kappa_{n}$  behave like
\begin{equation}\label{e5.21}
\begin{aligned}
\kappa_{\alpha}&=\frac{1}{w}\lambda_{\alpha}+o(1),\\
\kappa_{n}&=\frac{1}{w^{3}}u_{nn}(0)\Big(1+O(\frac{1}{|u_{nn}(0)|})\Big).
\end{aligned}
\end{equation}
as $|u_{nn}(0)|\rightarrow \infty$, where $\lambda_{1},\ldots, \lambda_{n-1}$ are the eigenvalues of $u_{\alpha\beta}(0)$.\par
Next we prove the following conclusion:
\begin{lemma}\label{t5.9}
We show that
 $$\max_{x\in\partial\Omega }|u_{nn}(x)|\leq C_{38}$$
  for all $x\in\partial\Omega$, where $C_{38}$  is  positive constant depending on $\partial\Omega, n, \delta ,h$ and $\underline{u}$.
\end{lemma}
\begin{proof}
Suppose that $|u_{nn}(x)|$ achieves its maximum at a point $x_{p}$ on $\partial\Omega$. We assume that $u_{nn}(x_{p})\geq  K$ where $K$ is a large constant to be chosen.
 Denote $\kappa_{\alpha}'=(\kappa'_{1},\ldots,\kappa'_{n-1})$  eigenvalues of  $a_{\alpha\beta}(x_{p})$. By (\ref{e5.21}), for every $\delta_{1}>0$, there exists $R(\delta_{1})\gg1$ such that if $u_{nn}(x_{p})\geq R(\delta_{1})$, then the eigenvalues of $a_{ij}(x_{p})$ satisfy
$$\kappa_{n}\geq R_{2},$$ and
     \begin{equation}\label{e5.22}
|(\kappa_{1},\ldots,\kappa_{n-1},\kappa_{n})-(\kappa_{1}',\ldots,\kappa'_{n-1},\kappa_{n})|<\delta_{1}.
  \end{equation}
  Consequently, we find that  there exists a positive constant $\delta_{1}>0$  depending on $\omega_{0}$ such that if $u_{nn}(x_{p})\geq K\triangleq\max\{ R(\delta_{1}), R_{2}\}$, then
  \begin{equation*}
  \begin{aligned}
\psi(x_{p})=G(\kappa)(x_{p})&\geq G(\kappa_{\alpha}',\kappa_{n})(x_{p})-\frac{\omega_{0}}{2}\\
&\geq G(\kappa_{\alpha}',R_{2})-\frac{\omega_{0}}{2}\\
&>\psi(x_{p})+\frac{\omega_{0}}{2},
\end{aligned}
  \end{equation*}
where  the first step follows from the continuity of  $G$,  the second step  follows from monotonicity of $G$, while the last comes from the fact that
Lemma \ref{t5.8}. So we obtain a contradictory result.
 Thus  we complete the proof of Lemma \ref{t5.9}.
\end{proof}
 Hence the boundary $C^{2}$ estimate is complete. Then this completes the proof of Lemma \ref{t4.1}.

\section{Proof of Theorem 1.3}
The goal of this section is to prove Theorem 1.3.    Since Lemma \ref{t4.1} and Lemma \ref{t4.2} have been established,  by  Evans-Krylov theorem and Schauder theory in \cite{GT}, we get the
priori bound for the $C^{3,\alpha}$  norm of $u$ for any $\alpha\in(0,1)$ to Dirichlet problem (\ref{e3.5}), i.e.
 \begin{equation}\label{e6.1}
\|u\|_{C^{3,\alpha}(\overline{\Omega})}\leq C_{39},
\end{equation}
where $C_{39}$ is a positive constant  depending  on $\Omega,\varphi,\|\underline{u}\|_{C^{4}(\overline{\Omega})},\|h\|_{C^{2}(\overline{\Omega})}$ and $\delta$.
 In this section we explain how to make use of (\ref{e6.1}) and continuity methold  to prove
   the existence of solution to (\ref{e1.3}) and (\ref{e1.4}).\par
    Following the same proof as Lemma  5.2 in \cite{CHY}, we can obtain
   \begin{lemma}\label{t6.1}
   If $\Omega,h,\varphi$ and $\underline{u}$ are smooth satisfying (A), then the solution  $u$ is smooth in $\bar{\Omega}$.
   \end{lemma}
   To solve the equation, we use the continuity method.\par

 \noindent{\bf Proof of Theorem 1.3.}

 It's enough for us to consider Dirichlet problem (\ref{e3.5}) as follows.
  \begin{equation*}
\left\{ \begin{aligned}
\widetilde{G}(D^{2}u, Du)&= \psi(x),\ \ && \quad
 \mathrm{in}\quad \Omega, \\
u&=\varphi,  && \quad
 \mathrm{on}\quad\partial\Omega.
\end{aligned} \right.
\end{equation*}
 where $\widetilde{G}(D^{2}u, Du)=G(\mathcal{A})$ and $\psi(x)=-e^{-Ah(x)}$.
 One can use the pre-knowledge in  previous section. Assume that  $\underline{u}$ is $C^{4}(\overline{\Omega})$ subsolution satisfying (\ref{e2.6}).
It yields that the classical solution to (\ref{e1.3}) and (\ref{e1.4}) is unique by maximum principle.
 For each $t\in[0,1]$,  set
 $$J^{t}(D^{2}u,Du)=\widetilde{G}(D^{2}u,tDu).$$
 Consider
  \begin{equation}\label{e6.2}
\left\{ \begin{aligned}
J^{t}(D^{2}u,Du)&=\psi(x),\ \ &&\mathrm{in}\quad \Omega, \\
u&=\varphi,             &&\mathrm{on}\quad \partial\Omega.
\end{aligned} \right.
\end{equation}
Define the set
$$S:=\left\{t\in[0,1]:\text{(\ref{e6.2}) has at least one admissible solution}\right\}.$$
By the main results in \cite{CPW},  (\ref{e6.2}) is solvable if $t=0$.
So $S$ is not empty. We claim that $S=[0,1]$, which is equivalent to the fact that $S$ is not only open, but also closed.

Define
$$X:=\{u\in C^{3,\alpha}(\bar{\Omega}):u\mid_{\partial\Omega}=\varphi\}$$
 and
$$Y:= C^{\alpha}(\bar{\Omega})\times C^{2,\alpha}(\partial\Omega).$$
Then $X$ is  the closed convex subset in Banach space $C^{3,\alpha}(\bar{\Omega})$ and $Y$ is a Banach space.
Define a map from $X\times [0,1]$ to $Y$ as
$$\mathfrak{F}(u,t):=\left(J^{t}(D^{2}u,Du)-\psi(x),u\right).$$
Given $(u_{0},t_{0})\in X\times [0,1]$. For  any $w\in C^{3,\alpha}(\overline{\Omega})$ with $w\mid_{\partial\Omega}=0$, we have
\begin{equation}\label{e6.3b}
\begin{aligned}
&\frac{d}{d\epsilon}\mathfrak{F}(u_{0}+\varepsilon w,t_{0}+\varepsilon)|_{\epsilon=0}\\
=&\left(\widetilde{G}_{ij}(D^2u_{0},t_{0}Du_{0})\partial_{ij}w+\widetilde{G}_{i}(D^2u_{0},t_{0}Du_{0})t_{0}\partial_{i}w+\widetilde{G}_{i}(D^2u_{0},t_{0}Du_{0})\partial_{i}u_{0}, 0\right).
\end{aligned}
\end{equation}
Then the linearized operator of $\mathfrak{F}(u,t)$ at $(u_{0},t_{0})$ is given by
$$D\mathfrak{F}(u_{0},t_{0})=\left(\widetilde{G}_{ij}(D^2u_{0},t_{0}Du_{0})\partial_{ij}w+\widetilde{G}_{i}(D^2u_{0},t_{0}Du_{0})t_{0}\partial_{i}w+\widetilde{G}_{i}(D^2u_{0},t_{0}Du_{0})\partial_{i}u_{0}, 0\right).$$
 One can see that $\widetilde{G}_{ij}(D^2u_{0},t_{0}Du_{0}), \widetilde{G}_{i}(D^2u_{0},t_{0}Du_{0})\in C^{1,\alpha}(\overline{\Omega})$ and
 $$\mathfrak{L}\triangleq \widetilde{G}_{ij}(D^2u_{0},t_{0}Du_{0})\partial_{ij}+\widetilde{G}_{i}(D^2u_{0},t_{0}Du_{0})t_{0}\partial_{i}$$
 is a  uniformly linear elliptic operator by calculation.  Then making use of Schauder theory in \cite{GT}, $D\mathfrak{F}(u_{0},t_{0})$ is invertible for any  $u_{0}\in X $ and  $t_{0}\in[0,1]$ being the solution to (\ref{e6.2}).
 Then the fact that S is open follows from the
invertibility of the linearized operator and the implicit function theorem \cite{GT}.

To finish the proof, we need to prove the fact that $S$ is a closed subset of $[0,1]$.
One can show that $S$ is closed is equivalent to the fact that for any sequence $\{t_k\}\subset S$, if $\lim_{k\rightarrow \infty}t_k=t_0$, then $t_0\in S$. For $t_k$, $u_{k}$ is  the solution of
  \begin{equation}\label{e6.3}
\left\{ \begin{aligned}
J^{t_{k}}(D^{2}u_{k}, Du_{k})&=\psi(x),\ \ &&\mathrm{in}\quad \Omega, \\
u_{k}&=\varphi,             &&\mathrm{on}\quad \partial\Omega.
\end{aligned} \right.
\end{equation}
  The operator $J^{t}$ comes from $\widetilde{G}$. By the proof of the main theorem in \cite{HL2},  the  structure conditions of $J^{t}$ is as same as $\widetilde{G}$. Then for any admissible solution of (\ref{e6.2}), the estimate (\ref{e6.2}) also holds. So any solution $u_{k}$ satisfies
$$\|u_{k}\|_{C^{3,\alpha}(\overline{\Omega})}\leq C_{40},$$
where $C_{40}$ depends only on the known data and also is  independent to $t_{k}$.
  Using Arzela-Ascoli Theorem, we know that there exist $\hat{u}\in C^{3,\alpha}(\bar{\Omega})$ and a subsequence of $\{t_k\}$, which is still denoted as $\{t_k\}$, such that letting $k\rightarrow \infty$,
\begin{equation*}
\left\|u_k-\hat{u}\right\|_{C^{3}(\bar{\Omega})}\rightarrow 0.
\end{equation*}
For equation (\ref{e6.3}),
letting $k\rightarrow \infty$, we have
\begin{equation*}
\left\{ \begin{aligned}
J^{t_{0}}(D^{2}\hat{u}, D\hat{u})&=\psi(x),\ \ &&\mathrm{in}\quad \Omega, \\
\hat{u}&=\varphi,             &&\mathrm{on}\quad \partial\Omega.
\end{aligned} \right.
\end{equation*}
Therefore, $t_0\in S$, and thus $S$ is closed. Consequently, $S=[0,1]$.

It follows Lemma 6.1  that there exists a  smooth solution $u\in C^{\infty}(\overline{\Omega})$ to the Dirichlet
problem (\ref{e1.3}) and (\ref{e1.4})  if all data is smooth.
Thus  Theorem \ref{t1.3} is established.
 \qed

{\bf Acknowledgments:}  The authors thank Professor Hengyu Zhou for helpful discussions and suggestions related to this work. The authors would also like to thank the referees for useful comments that have improved the paper.

\end{document}